\PassOptionsToPackage{nameinlink,noabbrev,capitalize}{cleveref}
\PassOptionsToClass{cleveref}{jmlr}
\documentclass[final,12pt]{colt2026} % Anonymized submission
\title[Empirical Bayes for Heteroskedastic Mean Estimation]{An Empirical Bayes Perspective on Heteroskedastic Mean Estimation}
\usepackage{times}

\makeatletter
\def\input@path{{../}{./}{../arxiv/}{../content/}}
\makeatother

\coltauthor{%
 \Name{Yanjun Han} \Email{yanjunhan@nyu.edu}\\
 \addr New York University
 \AND
 \Name{Abhishek Shetty} \Email{shetty@mit.edu}\\
 \addr Massachusetts Institute of Technology and Georgia Institute of Technology%
 \AND
 \Name{Jacob Shkrob} \Email{jas10184@nyu.edu}\\
 \addr New York University%
}

\usepackage{bbm}
\usepackage{commath}
\usepackage{bm}
\usepackage{xspace}
\usepackage{subcaption}

\usepackage{subcaption}
\usepackage{tikz}
\usetikzlibrary{arrows,calc,decorations.pathreplacing,matrix,positioning,patterns}
\usepackage{pgfplots}
\pgfplotsset{compat=1.18}
\usepackage{nicefrac}
\usepackage{comment}
\usepackage{mathtools}
\usepackage{prettyref}
\newtheorem{thm}{Theorem}[section]

\newtheorem{cor}{Corollary}[section]
\newtheorem{lem}{Lemma}[section]
\newtheorem{defn}{Definition}[section]
\newtheorem{rmk}{Remark}[section]

\AtBeginDocument{%
  \crefname{lemma}{Lemma}{Lemmas}%
  \Crefname{lemma}{Lemma}{Lemmas}%
  \crefname{corollary}{Corollary}{Corollaries}%
  \Crefname{corollary}{Corollary}{Corollaries}%
  \crefname{proposition}{Proposition}{Propositions}%
  \Crefname{proposition}{Proposition}{Propositions}%
  \crefname{theorem}{Theorem}{Theorems}%
  \Crefname{theorem}{Theorem}{Theorems}%
  \crefname{thm}{Theorem}{Theorems}%
  \Crefname{thm}{Theorem}{Theorems}%
  \crefname{assumption}{Assumption}{Assumptions}%
  \Crefname{assumption}{Assumption}{Assumptions}%
  \crefformat{equation}{(#2#1#3)}%
}

\DeclareMathOperator*{\argmax}{arg\,max}

\DeclareMathOperator{\supp}{supp}

\def\bP{\mathbb{P}}
\def\bE{\mathbb{E}}
\def\bR{\mathbb{R}}
\def\bC{\mathbb{C}}

\def\1{\mathbbm{1}}

\newcommand{\pth}[1]{\left( #1 \right)}
\newcommand{\qth}[1]{\left[ #1 \right]}
\newcommand{\sth}[1]{\left\{ #1 \right\}}
\newcommand{\bpth}[1]{\Big( #1 \Big)}
\newcommand{\bqth}[1]{\Big[ #1 \Big]}
\newcommand{\bsth}[1]{\Big\{ #1 \Big\}}

\newcommand{\rmd}{\mathrm{d}}

\newcommand{\Unif}{\text{\rm Unif}}

\newcommand{\indc}[1]{{\mathbf{1}_{\left\{{#1}\right\}}}}

\newcommand{\stepa}[1]{\overset{\rm (a)}{#1}}
\newcommand{\stepb}[1]{\overset{\rm (b)}{#1}}
\newcommand{\stepc}[1]{\overset{\rm (c)}{#1}}
\newcommand{\stepd}[1]{\overset{\rm (d)}{#1}}
\newcommand{\stepe}[1]{\overset{\rm (e)}{#1}}

\newcommand{\TV}{{\sf TV}}
\newcommand{\KL}{{\sf KL}}

\newcommand{\calF}{{\mathcal{F}}}
\newcommand{\calG}{{\mathcal{G}}}
\newcommand{\calH}{{\mathcal{H}}}

\newcommand{\calL}{{\mathcal{L}}}

\newcommand{\calN}{{\mathcal{N}}}

\newcommand{\calP}{{\mathcal{P}}}

\newcommand{\calX}{{\mathcal{X}}}

\newcommand{\muknownsig}{\widehat{\mu}_{\mathrm{known}-\sigma}}
\newcommand{\empmean}{\overline{\mu}}
\newcommand{\EB}{\mathrm{EB}}
\newcommand{\sigmin}{\sigma_{\min}}
\newcommand{\sigmax}{\sigma_{\max}}
\newcommand{\rL}{\mathrm{L}}
\newcommand{\rU}{\mathrm{U}}
\newcommand{\bracketing}{N_{[\,]}}

\newcommand{\R}{\mathbb{R}}

\newcommand{\AS}[1]{}
\newcommand{\JS}[1]{}

\begin{document}

\maketitle

\begin{abstract}
Towards understanding the fundamental limits of estimation from data of varied quality, we study the problem of estimating a mean parameter from heteroskedastic Gaussian observations where the variances are unknown and may vary across observations.
While, with known variances, a simple linear estimator attains the smallest mean squared error, estimation without this knowledge is challenging due to the large number of nuisance parameters.
We propose a simple and principled approach based on empirical Bayes: model the observations as if they were i.i.d. from a normal scale mixture and compute the profile maximum likelihood estimator (MLE) for the mean, treating the nonparametric mixing distribution as nuisance.
Our result shows that this estimator achieves near-optimal error bounds across various heteroskedastic models in the literature.
In particular, for the subset-of-signals problem where an unknown subset of observations has small variance, our estimator adaptively achieves the minimax rate for all signal sizes, including the sharp phase transition, without any tuning parameters.

One of our key technical steps is a sharper metric entropy bound for normal scale mixtures, obtained via generalized moment matching and Chebyshev approximation.
This approach yields an improved polylogarithmic, rather than polynomial, dependence on problem parameters, which could be of independent interest.
\end{abstract}


%%%%%%%%%%%%%%%%%%%%%%%%%%%%%%%%%%%%%%%%%%%%%%%%%%%%%%%%%%%%%%%%%
\section{Introduction}
%%%%%%%%%%%%%%%%%%%%%%%%%%%%%%%%%%%%%%%%%%%%%%%%%%%%%%%%%%%%%%%%%

Estimation of a signal from heterogeneous data lies at the heart of statistics and machine learning.
The heterogeneity can arise from variations in data quality, measurement precision, or sampling conditions and is ubiquitous in real-world applications. 
Thus, it is a fundamental statistical challenge to develop principled methods that allow one to estimate signals robustly from heterogeneous data with minimal knowledge of the data quality.

{Perhaps the simplest} abstraction of this general objective is estimation of a single one dimensional parameter from data of varying quality.
In particular, given independent observations $X_1,\dots,X_n$ with $X_i\sim \calN(\mu,\sigma_i^2)$, the learner's target is to estimate the mean parameter $\mu\in \bR$. 
Even in this simple one dimensional setting, heterogeneity leads to statistical challenges.
%  because of its heterogeneous nature, i.e., the variance parameters $\sigma_i^2$ are \emph{unknown} and could vary across different observations.
In the case {when} $(\sigma_1,\dots,\sigma_n)$ are known, the maximum likelihood estimator (MLE) for $\mu$ is
\begin{align}\label{eq: mle known si}
\muknownsig = \frac{\sum_{i=1}^n \frac{X_i}{\sigma_i^2} }{\sum_{i=1}^n \frac{1}{\sigma_i^2}}, 
\end{align}
which is also the uniformly minimum-variance unbiased estimator (UMVUE) and achieves the optimal error rate of $\bE[\abs{\widehat{\mu}-\mu}] =  (\sum_{i=1}^n \frac{1}{\sigma_i^2})^{-1/2} $.
{The key aspect of this estimator} is that this bound is \emph{robust} to large variances. 
For example, if there are $n/2$ variances that are large while the other $n/2$ variances are small, then the error rate is still $O(1/\sqrt{n})$, as the large variances are effectively ignored in the weighted average.
Compare this to the error of the sample average $ \empmean = \frac{1}{n}\sum_{i=1}^n X_i$, which has error rate $\bE[( \empmean - \mu)^2] = \frac{1}{n^2}\sum_{i=1}^n \sigma_i^2$ and could be arbitrarily large even if a single variance is large. 
On the other hand, $ \muknownsig $ critically relies on the knowledge of $(\sigma_1,\dots,\sigma_n)$, which could not be estimated reliably simply because the number of unknown parameters $n+1$ is more than the number of observations $n$ while the sample average $\empmean $ is oblivious to the knowledge of $(\sigma_1,\dots,\sigma_n)$. 
The key question that we would like to address is:

\begin{quote}
    Are there general principled methods that allow one to estimate signals robustly from heterogeneous data with minimal knowledge of the data quality?
\end{quote}

% However, this estimator critically relies on the knowledge of $(\sigma_1,\dots,\sigma_n)$, which could not be estimated reliably simply because the number of unknown parameters $n+1$ is more than the number of observations $n$. If one na\"ively use the sample average $\widehat{\mu} = \overline{X}_n$, this estimator is not robust to outliers in the data with very large $\sigma_i$. 

% To fix this issue, the literature turns to robust estimators such as sample median and many of its variants. 

An important special case considered in the literature is the ``subset of signals'' problem \citep{liang2020learning}, where it is assumed that $|\sth{i\in [n]: \sigma_i\le 1}| \ge m$, i.e., $m$ out of $n$ observations can be treated as ``signals'' (although the location of this subset of signals is still unknown).
Even for this simple setting, the minimax rate has only been recently characterized \citep{compton2024near}:
% This problem is well understood, with the minimax mean squared error (MSE) characterized \citep{}: 
there exists an estimator $\widehat{\mu}$ with
% TODO: Add figure illustrating minimax rates showing phase transition at m ~ n^{1/4}
\begin{align}\label{eq:minimax_risk}
\bE |\widehat{\mu} - \mu| = \begin{cases}
    \widetilde{O}\bpth{\bpth{\frac{n}{m^4}}^{1/2}} &\text{if }\log n\ll m\le n^{1/4}, \\
    \widetilde{O}\bpth{\bpth{\frac{n}{m^4}}^{1/6}} & \text{if } n^{1/4}\le m\le n, 
\end{cases}
\end{align}
and this dependence on $(m,n)$ (including the phase transition at $m\asymp n^{1/4}$) is not improvable in the worst case.\footnote{Here and throughout, $\widetilde{O}(\cdot)$ hides polylogarithmic factors in $n$.} While the minimax rate is known, existing estimators \citep{compton2024near} achieving this minimax rate are often quite complicated, requiring multiple tunable parameters and are reliant on the particular formulation of the problem, and tend not to lead to unifying statistical principles for learning from heterogeneous data.
% Although this minimax rate of convergence has been fully characterized, the proposed estimators often has unwieldy features: they require numerous tuning parameters and rely on knowledge of $m$ \YH{check}, which is frequently unavailable in practice. Moreover, at a high level, existing optimal estimators are typically constructed in a case-by-case manner, rather than arising from a unified statistical principle.

Towards building a unified theory of heteroskedasticity in estimation, we make a perhaps surprising connection to the study of \emph{empirical Bayes} methods.
Using this connection, we show that the simple principle of maximum likelihood estimation, in conjunction with an empirical Bayes framework, leads to a unified estimator that achieves the minimax rate adaptively for all ranges of $m$. 
% In this paper, we provide a surprising redemption of the MLE principle with the help of \emph{empirical Bayes}: 
In this framework, instead of modeling the heterogeneous variances $\sigma_1,\dots,\sigma_n$ separately, we model them using their empirical distribution $G_n := \frac{1}{n}\sum_{i=1}^n \delta_{\sigma_i}$ and treat the observations $X_1,\dots,X_n$ \emph{as if} they were \emph{i.i.d.} drawn from a normal scale mixture $f_{\mu,G_n}(x) := \bE_{\sigma\sim G_n}[\frac{1}{\sigma}\varphi(\frac{x-\mu}{\sigma})]$, where $\varphi$ denotes the standard normal density. 
% \footnote{Here $\phi$ refers to the Gaussian density $\phi(x) = 1/ \sqrt{2 \pi } \cdot \exp(-x^2/2)$}. 
Under this i.i.d. model, the joint MLE $(\widehat{\mu}, \widehat{G})$ is naturally defined as 
\begin{align}\label{eq:EB_MLE}
(\widehat{\mu}, \widehat{G}) = \argmax_{\mu\in \bR, \supp(G)\subseteq [\sigmin, \sigmax]} \frac{1}{n}\sum_{i=1}^n \log f_{\mu, G}(X_i), 
\end{align}
and the final estimator is simply defined as $\widehat{\mu}^{\EB} = \widehat{\mu}$ in \eqref{eq:EB_MLE}. In other words, our estimator $\widehat{\mu}^{\EB}$ is the profile MLE \citep{murphy2000profile} of $\mu$, treating $G$ as a nuisance parameter. Note that the nuisance estimator $\widehat{G}$ in \eqref{eq:EB_MLE} is a \emph{nonparametric} MLE (NPMLE), since it may be any probability distribution supported on $[\sigmin, \sigmax]$. In particular, this formulation no longer enforces the structure of an empirical distribution, nor does it explicitly require the ``subset-of-signals'' structure $G([0,1]) \ge \frac{m}{n}$. Here $\sigmin, \sigmax > 0$ are hyperparameters used in the MLE and satisfy $\sigmin \le \sigma_i \le \sigmax$ for every $i \in [n]$; further discussion is provided later. 

\subsection{Main results}
Our first result shows that, although the estimator $\widehat{\mu}^{\EB}$ is developed from an entirely different principle compared with existing estimators, it nevertheless has an instance-dependent error bound.

\begin{thm}\label{thm:general}
Let $\sigma_i\in [\sigmin, \sigmax]$ for all $i\in [n]$. With probability at least $1-\delta$, the estimator $\widehat{\mu}^{\EB}$ in \eqref{eq:EB_MLE} achieves (the exact logarithmic factor is displayed in \Cref{lemma:density_estimation})
\begin{align*}
|\widehat{\mu}^{\EB} - \mu| \le C\omega_{H^2, G_n}\bpth{\frac{\mathrm{polylog}(n, \frac{\sigmax}{\sigmin}, \frac{1}{\delta})}{n}}, 
\end{align*}
where $\omega_{H^2,G_n}(t)$ is the Hellinger modulus of continuity in the location family: 
\begin{align*}\label{eq:modulus-of-continuity}
\omega_{H^2,G_n}(t) = \sup\sth{ |\mu_1-\mu_2|: \mu_1, \mu_2\in \bR, H^2(f_{\mu_1,G_n}, f_{\mu_2,G_n}) \le t }. 
\end{align*}
\end{thm}

We remark that \Cref{thm:general} establishes an upper bound that is competitive with the best oracle estimator possessing knowledge of $G_n$, uniformly over all possible choices of $G_n$. Indeed, even in the mean estimation problem where $X_1,\dots,X_n$ are i.i.d. drawn from $f_{\mu, G_n}$ with \emph{known} $G_n$, Le Cam's two-point method \citep{lecam1973convergence,lecam1986asymptotic} yields a minimax lower bound of $\omega_{H^2,G_n}(\frac{1}{n})$.\footnote{This precise argument does not hold in the compound setting, but this intuition is conjectured to remain valid. See \Cref{sec:conclusion} for detailed discussion.} %\textcolor{red}{See \citep{vaart1998asymptotic} for further discussion of the precise form of Le Cam's method in the compound setting.}}
By comparison, the radius in the Hellinger modulus of continuity in our upper bound is $\widetilde{O}(\frac{1}{n})$, which essentially matches this lower bound even when $G_n$ is unknown.

Specializing to several heterogeneous mean estimation models studied in the literature \citep{pensia2022estimating,devroye2023mean}, \Cref{thm:general} implies the following error bounds, all of which match the best known rates (\citep[Table 1]{pensia2022estimating} and \citep[Section 5.1]{devroye2023mean}) up to logarithmic factors. 
\begin{cor}\label{cor:examples}
Let $L = \widetilde{O}(1)$ be the poly-logarithmic factor in \Cref{thm:general}. The following high-probability guarantee holds in specific examples: 
\vspace{-0.5em}
\begin{enumerate}
    \item Equal variance: when $\sigma_i \equiv 1$, then $|\widehat{\mu}^{\EB} - \mu| = O(\sqrt{\frac{L}{n}})$; 
    \vspace{-0.8em}
    \item Quadratic variance: when $\sigma_i = i$, then $|\widehat{\mu}^{\EB} - \mu| = O(L\log n)$; 
    \vspace{-0.5em}
    \item Two variances: when $\sigma_1=\dots=\sigma_m=1$ and $\sigma_{m+1}=\cdots=\sigma_n=\sigma\ge 1$ with $m\le \frac{n}{2}$, then $|\widehat{\mu}^{\EB} - \mu| = O(\sigma\sqrt{\frac{L}{n}})$. For specific ranges of $m$, sharper upper bounds can be obtained: 
    \begin{align*}
    |\widehat{\mu}^{\EB} - \mu| = \begin{cases}
        O(\frac{\sqrt{nL}}{m}) &\text{if } m \ge \sqrt{nL}, \\
        O(1) &\text{if } \sqrt{\frac{nL}{\sigma}} + L \le m < \sqrt{nL}. 
    \end{cases}
    \end{align*}
    \vspace{-1em}
    \item $\alpha$-mixture distribution: when $\sigma_1=\dots=\sigma_m = 1$ and $\sigma_{m+1}=\dots=\sigma_n=n^{\alpha}$ with $m=cL$, then $|\widehat{\mu}^{\EB} - \mu| = O(n^{\alpha-1/2}\sqrt{L})$ if $0<\alpha<1$ and $|\widehat{\mu}^{\EB} - \mu| = O(1)$ if $\alpha\ge 1$. 
\end{enumerate}
\end{cor}

Finally, for the subset-of-signals problem with $m$ signals, the estimator $\widehat{\mu}^{\EB}$ achieves near-optimal error bounds, including the correct phase transition, across all ranges of $m$: 
\begin{thm}\label{thm:main}
Let $\sigma_i\in [\sigmin, \sigmax]$ for all $i\in [n]$, and the number of signals be $m$ in the subset-of-signals problem. With probability at least $1-\delta$, estimator $\widehat{\mu}^{\EB}$ in \eqref{eq:EB_MLE} achieves  
\begin{align*}
|\widehat{\mu}^{\EB} - \mu| = \begin{cases}
    \widetilde{O}\bpth{\bpth{\frac{n}{m^4}}^{1/2}} &\text{if }\widetilde{O}(1)\ll m\le n^{1/4}, \\
    \widetilde{O}\bpth{\bpth{\frac{n}{m^4}}^{1/6}} & \text{if } n^{1/4}\le m\le n, 
\end{cases}
\end{align*}
where $\widetilde{O}(\cdot)$ hides polylogarithmic factors in $(n, \frac{\sigmax}{\sigmin}, \frac{1}{\delta})$. 
\end{thm}

Compared with the known minimax risk \eqref{eq:minimax_risk} for the subset-of-signals problem, \Cref{thm:main} shows that the estimator $\widehat{\mu}^{\EB}$ achieves the minimax rate up to logarithmic factors under the mild assumption $\log\frac{\sigmax}{\sigmin} = O(\mathrm{polylog}(n))$. However, like many empirical Bayes approaches, a particularly appealing advantage of our estimator is that it is \emph{adaptive} and \emph{parameter-free}: this estimator requires no tuning parameters (we set $\sigma_{\min}= 0$ and $\sigma_{\max}=\infty$ in our experiments; see \Cref{sec:numerics}), and adapts to the signal size $m$. These features make $\widehat{\mu}^{\EB}$ especially attractive in practice. In the practical scenario where the joint MLE \eqref{eq:EB_MLE} is computed only approximately, a similar guarantee remains valid for approximate MLEs; see \Cref{remark:approximate_MLE}.

Compared with existing empirical Bayes approaches, while the idea of applying an NPMLE-based estimator is classical in the empirical Bayes literature \citep{kiefer1956consistency,robbins1956empirical}, our objective in \eqref{eq:EB_MLE} is conceptually different. In existing theoretical studies of empirical Bayes, one typically considers independent observations $X_i \sim P_{\theta_i}$ (the compound setting) and aims to estimate the parameter vector $(\theta_1,\dots,\theta_n)$. The empirical Bayes approach proceeds by first estimating the empirical distribution $G_n := \frac{1}{n}\sum_{i=1}^n \delta_{\theta_i}$ (e.g., via the NPMLE), and then applying the resulting learned prior to estimate $\theta$ through the corresponding Bayes rule. By contrast, our procedure likewise replaces heterogeneity in the \emph{nuisance parameter} $(\sigma_1,\dots,\sigma_n)$ with a mixing distribution $G_n$ and learns it from the data, but does not invoke explicit posterior inference. It is therefore an interesting theoretical observation that empirical Bayes methods are near-optimal in our setting, even though the procedure does not actually invoke a ``Bayes'' component. 

We also provide some discussion on the hyperparameters $(\sigmin, \sigmax)$ and computation. We acknowledge that our estimator requires additional knowledge of $(\sigmin, \sigmax)$, which is not needed by previous estimators. This requirement is, however, intrinsic to our approach: without it, the solution to \eqref{eq:EB_MLE} would be degenerate, since one could assign a positive mass to $\widehat{G}(\{0\})$ and choose $\widehat{\mu}\in \sth{X_1,\dots,X_n}$, thereby making the likelihood unbounded. Fortunately, the error dependence on $\sigma_{\min}$ in \Cref{thm:main} is only $\mathrm{polylog}(\frac{ \sigma_{\max} }{\sigma_{\min}})$ and thus very mild, and our numerical experiments show that even setting $\sigmin=0$ and $\sigmax=\infty$ still yields sound estimation performance. As for computation, solving the optimization program in \eqref{eq:EB_MLE} is, unfortunately, a nonconvex problem involving the infinite-dimensional parameter $\widehat{G}$. In \Cref{sec:numerics}, we adopt a successive maximization strategy. For fixed $\mu$, \eqref{eq:EB_MLE} reduces to a convex optimization problem in $G$, for which it is standard to apply a fully corrective Frank--Wolfe algorithm. For fixed $G$, \eqref{eq:EB_MLE} becomes a one-dimensional maximization problem in $\mu$, which can be efficiently solved via grid search. Although both Frank--Wolfe and successive maximization may converge to local maxima, our numerical results indicate that the resulting $\widehat{\mu}^{\EB}$ is empirically close to the ideal MLE obtained when $G_n$ is known exactly.

\subsection{Outline of the proof}
\Cref{thm:general} follows from a key result concerning the density estimation performance of the MLE, which accurately estimates the true mixing density $f_{\mu, G_n}$ in Hellinger distance. 

% \paragraph{Error of Density Estimation}

\begin{thm}\label{lemma:density_estimation}
Let $\sigmin\le \sigma_i\le \sigmax$ for all $i\in [n]$, and $(\widehat{\mu},\widehat{G})$ be the joint MLE in \eqref{eq:EB_MLE}. Then with probability at least $1-\delta$, 
\begin{align*}
H^2(f_{\mu,G_n}, f_{\widehat{\mu}, \widehat{G}}) \le \frac{C}{n}\pth{\log^6\bpth{\frac{n\sigmax}{\sigmin}}\log\log\bpth{\frac{n\sigmax}{\sigmin}} + \log\frac{1}{\delta}}. 
\end{align*}
\end{thm}

\begin{rmk}\label{remark:approximate_MLE}
If $(\widehat{\mu},\widehat{G})$ is an approximate MLE with
% \begin{align*}
$\prod_{i=1}^n f_{\widehat{\mu},\widehat{G}}(X_i) \ge \beta\cdot \max_{\mu, G }\prod_{i=1}^n f_{\mu,G}(X_i)$
% \end{align*}
with $\beta\in (0,1]$, by \Cref{lemma:entropic-upper-bound}, the Hellinger bound in \Cref{lemma:density_estimation} has an additive factor $O(\frac{1}{n}\log\frac{1}{\beta})$. 
\end{rmk}

\Cref{lemma:density_estimation} establishes an upper bound in Hellinger distance for estimating the density $f_{\mu,G_n}$, which is the true average distribution of $X_1,\dots,X_n$. 
Obtaining density estimation guarantees under the Hellinger metric is standard in the statistical literature for analyzing the MLE, thanks to its well-known connection with the metric entropy of the underlying density class \citep{wong1995probability,geer2000empirical}. \Cref{thm:general} is then a direct consequence of \Cref{lemma:density_estimation} and a symmetrization inequality for the Hellinger distance in \Cref{lemma:symmetrization}, presented in \Cref{sec:symmetrization}.

\paragraph{Upper bound on the covering number}
In the proof of \Cref{lemma:density_estimation}, a key technical step is to analyze the metric entropy of the family of normal scale mixtures $\sth{f_{0, G}: \mathrm{supp}(G)\subseteq [\sigmin,\sigmax]}$ under the Hellinger metric.
Although such normal scale mixtures have been studied in the literature \citep{ghosal2001entropies,ghosal2007posterior,ignatiadis2025empirical}, the existing arguments based on (local) moment matching \citep{ghosal2001entropies,zhang2009generalized} yield metric entropy bounds with a dependence of $\mathrm{poly}(\frac{\sigmax}{\sigmin})$; by contrast, we critically need to improve this dependence to $\mathrm{polylog}(\frac{\sigmax}{\sigmin})$. 
To achieve this, we introduce two simple yet effective ideas. 
\vspace{-0.5em}
\begin{itemize}
    \item First, instead of applying classical moment matching based on polynomials, we perform a generalized moment matching by constructing a suitable basis $\{a_k(\sigma), g_k(x)\}$ to approximate the normal density $\frac{1}{\sigma}\varphi(\frac{x}{\sigma})$ by a separable expansion $\sum_{k=1}^L a_k(\sigma)g_k(x)$.
    \vspace{-0.5em}
    \item Second, instead of using the usual Taylor approximating polynomial for a given analytic function, we employ Chebyshev polynomials, which more accurately capture growth behavior on a Bernstein ellipse rather than on a disk in the complex plane.
\end{itemize}

\vspace{-0.5em}
These arguments establish the $\mathrm{polylog}(\frac{\sigmax}{\sigmin})$ dependence, which we elaborate in \Cref{sec:density_est}.
% and we defer the details to \Cref{sec:density_est}. 

% \begin{proof}[Proof of \cref{thm:general}]
% Let $(\widehat{\mu}, \widehat{G})$ be the joint MLE in \eqref{eq:EB_MLE}.
%     From \Cref{lemma:symmetrization}, we have $   $ 
%     % From \Cref{lemma:density_estimation} we have with probability at least $1-\delta$, that $ (\hat{mu} , \hat{G}) $ satisfies $$    
% \end{proof}

\paragraph{Upper bounding Hellinger modulus of continuity}

To prove \Cref{thm:main}, we need to upper bound the Hellinger modulus of continuity $\omega_{H^2,G}(t)$.  
% perhaps under additional structural assumptions on $G$.
% when $G([0,1])\ge p$. 
We prove such a bound in the case when $G([0,1])\ge p$ for some $p>0$, which captures instances such as the subset-of-signals problem.

\begin{lem}\label{lemma:functional_inequality}
Let $t\in [0,1]$, and $G$ be a prior distribution over $[0,\infty)$ with $G([0,1])\ge p > 0$. Then for a universal constant $C>0$, 
\begin{align}
\omega_{H^2,G}(t) \le C\begin{cases}
    (\frac{t^3}{p^4})^{1/6} &\text{if } t \le p^{4/3}, \\
    (\frac{t^3}{p^4})^{1/2} &\text{if } p^{4/3} < t \le \frac{p}{C}. 
\end{cases}
\end{align}
\end{lem}

\Cref{lemma:functional_inequality} is a functional inequality that exploits the structure of normal scale mixtures and will be proved in \Cref{sec:func_ineq_overall}. 
We note that this inequality is not improvable in general, as witnessed by the example $G=p\delta_1+(1-p)\delta_{c\mu/\sqrt{t}}$, with $\mu=\mu(p,t)$ matching the upper bound in \Cref{lemma:functional_inequality}; we verify the algebra in \Cref{sec:func_ineq_overall}. The condition $t=O(p)$ is also necessary: for the above choice of $G$, the same \Cref{sec:func_ineq_overall} verifies that $\omega_{H^2,G}(t)$ can be unbounded for $t\ge 3p$. 
We note that \Cref{thm:main} follows directly from \Cref{thm:general} and \Cref{lemma:functional_inequality}, with the subset-of-signals assumption giving $G_n([0,1])\ge \frac{m}{n} =: p$. 
\cref{cor:examples} similarly follows from \Cref{thm:general} and specific calculations of $\omega_{H^2,G_n}(t)$ for the respective choices of $G_n$, which we defer to \Cref{sec:mod_cont}.

\section{Related work}\label{app:related-work}

\paragraph{Heteroskedastic mean estimation.}
The problem of estimating a common location parameter from heterogeneous observations has attracted significant recent attention in both the statistics and theoretical computer science communities. Pensia, Jog, and Loh~\citep{pensia2022estimating} initiated the systematic study of this problem, characterizing minimax rates for location estimation under sample-heterogeneous distributions and proposing estimators based on iterative trimming and median-of-means techniques. Devroye, Lattanzi, Lugosi, and Zhivotovskiy~\citep{devroye2023mean} further developed this line of work, obtaining sharp minimax bounds for heteroskedastic mean estimation under bounded variance assumptions and establishing the phase transition at $m \asymp n^{1/4}$ in the subset-of-signals regime. Both works reveal a fundamental tension in this setting: the sample mean is efficient when variances are known but highly sensitive to large variances, while robust alternatives such as the median~\citep{huber1964robust,hampel1986robust} are resilient but often suboptimal.

The heteroskedastic setting studied here differs from the classical Huber contamination model~\citep{huber1964robust,huber1981robust}, where an $\varepsilon$-fraction of observations are adversarially corrupted. In the contamination model, robust estimators aim for error scaling with $\varepsilon$~\citep{diakonikolas2019robust,lai2016agnostic,chen2018robust}, whereas in our setting the noise levels can vary continuously across samples, leading to different minimax rates and requiring distinct estimation strategies.

\paragraph{Robust statistics and breakdown point.}
Classical robust statistics provides a rich toolkit for handling outliers and heavy-tailed distributions. Foundational work by Tukey~\citep{tukey1960survey,tukey1975mathematics}, Huber~\citep{huber1964robust,huber1981robust}, and Hampel~\citep{hampel1971general,hampel1986robust} established the concepts of breakdown point, influence function, and M-estimation that underpin modern robust methods. The median achieves the optimal breakdown point of $1/2$, meaning it remains bounded even when nearly half the observations are arbitrarily corrupted. Trimmed means~\citep{bickel1965some,stigler1973simon} offer a tunable tradeoff between efficiency and robustness. However, these classical estimators do not achieve the optimal rates in the heteroskedastic setting characterized by~\citep{pensia2022estimating,devroye2023mean}, motivating the search for new approaches.

\paragraph{Empirical Bayes and compound decision theory.}
Empirical Bayes methodology, pioneered by Robbins~\citep{robbins1951asymptotically,robbins1956empirical}, provides a principled framework for estimation problems with many latent parameters. In the compound decision setting, one observes $X_i \sim P_{\theta_i}$ for $i = 1,\ldots,n$ and aims to estimate the parameter vector $(\theta_1,\ldots,\theta_n)$. The empirical Bayes approach treats the parameters as draws from an unknown prior $G$ and estimates $G$ from the marginal distribution of the observations. The nonparametric maximum likelihood estimator (NPMLE) for the mixing distribution, introduced by Kiefer and Wolfowitz~\citep{kiefer1956consistency}, plays a central role in this framework. Laird~\citep{laird1978nonparametric} developed practical algorithms for computing the NPMLE, and Lindsay~\citep{lindsay1983geometry} established that it is a discrete measure with at most $n$ support points. More recently, Polyanskiy and Wu~\citep{polyanskiy2020self} discovered a remarkable self-regularization property of the NPMLE: even without explicit regularization, the NPMLE for subgaussian mixtures has $O(\log n)$ support points with high probability. The REBayes package~\citep{koenker2014convex,koenker2017rebayes} provides efficient implementations of NPMLE-based procedures.

A major breakthrough in the theoretical study of empirical Bayes was made by Jiang and Zhang \citep{jiang2009general} in the Gaussian compound decision problem. This work inspired a flurry of subsequent studies in Gaussian \citep{brown2009nonparametric,saha2020nonparametric,polyanskiy2021sharp,ghosh2025stein} and Poisson models \citep{brown2013poisson,polyanskiy2021sharp,shen2022empirical,jana2023empirical,jana2025optimal,han2025besting}. We also refer to a recent book \citep{efron2024empirical} for an overview of empirical Bayes approaches. 

Our approach differs from classical empirical Bayes in an important way: while standard empirical Bayes treats the heterogeneous parameters $(\sigma_1,\ldots,\sigma_n)$ as the primary estimation target and applies Bayes rules, we treat them as nuisance parameters and focus on estimating the common mean $\mu$ through a profile likelihood~\citep{murphy2000profile,severini2000likelihood}. This perspective connects our work to the literature on profile likelihood inference with infinite-dimensional nuisance parameters~\citep{murphy1997semiparametric,shen1997methods}. Nevertheless, our approach is still based on an important idea of empirical Bayes, where we replace heterogeneity by a mixing distribution and learn it from data. 

\paragraph{Density estimation and sieve MLE theory.}
The analysis of our estimator relies on density estimation theory for the MLE in mixture models. The foundational work of Wong and Shen~\citep{wong1995probability} established probability inequalities for likelihood ratios and convergence rates for sieve MLEs, showing that the rate is governed by the Hellinger metric entropy of the density class. Van de Geer~\citep{geer2000empirical} developed a comprehensive theory of M-estimation with empirical processes. The similar idea is used in our entropic upper bound in \Cref{lemma:entropic-upper-bound}. 

Specializing to the MLE in mixture models, its density estimation performance under the Hellinger distance has been studied in normal location mixture \citep{ghosal2001entropies,ghosal2007posterior,zhang2009generalized}, normal scale mixture \citep{ghosal2001entropies,ghosal2007posterior,ignatiadis2025empirical}, and Poisson mixture \citep{shen2022empirical,jana2025optimal} models. Existing approaches are mainly based on the idea of (local) moment matching and polynomial approximations~\citep{birge1998minimum,genovese2000rates}, usually leading to $\mathsf{polylog}(\frac{1}{\varepsilon})$ dependence on the radius $\varepsilon$ in the metric entropy bound. However, the existing dependence on $\frac{\sigmax}{\sigmin}$ in the normal scale mixture is not tight, and our analysis improves the dependence to logarithmic using generalized moment matching and Chebyshev polynomial approximations, which better capture the analytic structure of the Gaussian kernel. This improvement is crucial for obtaining minimax-optimal rates that are robust to the ratio $\frac{\sigmax}{\sigmin}$. 

\paragraph{Modulus of continuity and minimax lower bounds.}
Our upper bounds are expressed in terms of the Hellinger modulus of continuity of the location family, following the geometric approach to minimax theory developed by Donoho and Liu~\citep{donoho1991geometrizing2}. Le Cam's two-point method~\citep{lecam1973convergence,lecam1986asymptotic} provides matching lower bounds: if $H^2(f_{\mu_1,G}, f_{\mu_2,G}) \le 1/n$, then one cannot reliably distinguish $\mu_1$ from $\mu_2$, yielding a lower bound of $\omega_{H^2,G}(1/n)$ for estimating $\mu$. Therefore, \Cref{thm:general} can be viewed as a duality result to Le Cam's lower bound; such duality has appeared in the literature for linear functionals \citep{donoho1991geometrizing2,juditsky2009nonparametric,polyanskiy2019dualizing} and more recently for mean estimation in location families \citep{compton2025attainability}. There the mixing distribution $G_n$ is known, and the estimator is of a different Birg\'e--Le Cam type. In comparison, our empirical Bayes perspective shows that the profile MLE attains a comparable bound even when $G_n$ is unknown. 

% The functional inequality in \Cref{lemma:functional_inequality} characterizes this modulus for scale mixtures, capturing the phase transition in the minimax rate.

\paragraph{Computation.}
The optimization problem in \Cref{eq:EB_MLE} is nonconvex due to the joint optimization over $(\mu, G)$, but for fixed $\mu$ reduces to a convex problem in $G$. Algorithms for computing the NPMLE in mixture models include the EM algorithm~\citep{dempster1977maximum,laird1978nonparametric}, interior point methods~\citep{lesperance1992algorithm}, and the Frank--Wolfe algorithm~\citep{frank1956algorithm,jaggi2013revisiting}. Koenker and Mizera~\citep{koenker2014convex} reformulated the NPMLE computation as a convex optimization problem, enabling efficient solutions for large-scale problems. Our implementation uses a successive maximization strategy that alternates between Frank--Wolfe updates for $G$ and line search for $\mu$; similar block coordinate approaches are common in semiparametric estimation~\citep{murphy1999current,vaart1998asymptotic}.

% Nevertheless, our approach is still based on an important idea of empirical Bayes, where we replace heterogeneity by a mixing distribution and learn it from data. 

% Normal EB: 
% Poisson EB: 
% The key connection between 

% Survey 

%%%%%%%%%%%%%%%%%%%%%%%%%%%%%%%%%%%%%%%%%%%%%%%%%%%%%%%%%%%%%%%%%
% \section{Preliminaries}
%%%%%%%%%%%%%%%%%%%%%%%%%%%%%%%%%%%%%%%%%%%%%%%%%%%%%%%%%%%%%%%%%

%%%%%%%%%%%%%%%%%%%%%%%%%%%%%%%%%%%%%%%%%%%%%%%%%%%%%%%%%%%%%%%%%
\section{Density Estimation: Proof of \texorpdfstring{\Cref{lemma:density_estimation}}{Lemma}}\label{sec:density_est}
%%%%%%%%%%%%%%%%%%%%%%%%%%%%%%%%%%%%%%%%%%%%%%%%%%%%%%%%%%%%%%%%%

% \input{prelims.tex}

In this section we prove the density estimation guarantee of the MLE in \Cref{lemma:density_estimation}. We begin with a general entropic upper bound of the Hellinger distance in the compound (or i.n.i.d.) setting, and then present a metric entropy upper bound for the normal scale mixture with a fixed mean. The remaining details will be devoted to allowing $\mu$ to vary; we defer them to the appendix. 

\subsection{Entropic Upper Bound of the MLE in the Compound Setting}
In this section we review and state the classical entropic upper bound for density estimation via the MLE, in a general i.n.i.d. (independent and \emph{non-identically} distributed) setting. Let $X_1,\dots,X_n$ be independent, with $X_i\sim P_i$. Given a class $\calP$ of distributions such that $\overline{P}:=\frac{1}{n}\sum_{i=1}^n P_i \in \calP$, let $\widehat{P}$ be a $\beta$-approximate MLE with $
\prod_{i=1}^n \widehat{P}(X_i) \ge \beta \cdot \max_{P\in \calP} \prod_{i=1}^n P(X_i).$
To bound the Hellinger distance $H(\widehat{P}, \overline{P})$, we need the following notation for the metric entropy. For $\delta>0$, define the localized family $
\calP(\overline{P}, \delta) := \{ P\in \calP: H(P, \overline{P}) \le \delta \}. $

\begin{defn}[Bracketing Number]
The ($\delta$-local) Hellinger bracketing number at scale $\varepsilon$ around $\overline{P}$, denoted by $\bracketing(\varepsilon, \calP(\overline{P},\delta), H)$, is the smallest integer $N$ such that there exist $N$ brackets of (positive) measures $[P_1^{\rL}, P_1^{\rU}],\dots,[P_N^{\rL}, P_N^{\rU}]$ such that $\max_{i\in [N]} H(P_i^{\rL}, P_i^{\rU})\le \varepsilon$,
% \footnote{Here the definition of $H(\mu,\nu):=\int (\sqrt{\rmd \mu/\rmd \gamma}-\sqrt{\rmd \nu/\rmd \gamma})^2 \rmd \gamma$ is extended to non-probability measures in the natural way, with any dominating measure $\gamma$.} 
and for every $P\in \calP(\overline{P},\delta)$, there exists $i\in [N]$ such that $P_i^{\rL} \le P \le P_i^{\rU}$ (with $\le$ being pointwise).
\end{defn}

In the above definition, we extend the definition of Hellinger distance between two probability measures in a natural way to general measures $\mu, \nu$ by $H^2(\mu,\nu) := \int (\sqrt{\rmd \mu}-\sqrt{\rmd \nu})^2$. 
% The (local) Hellinger bracketing number $N_{[]}(\varepsilon, \calP(\delta), H)$ is the smallest integer $N$ with the following property: there exist $N$ brackets of (non-probability) measures $[P_1^{\rL}, P_1^{\rU}],\dots,[P_N^{\rL}, P_N^{\rU}]$ such that $\max_{i\in [N]} H(P_i^{\rL}, P_i^{\rU})\le \varepsilon$,\footnote{Here the definition of $H(\mu,\nu):=\int (\sqrt{\rmd \mu/\rmd \gamma}-\sqrt{\rmd \nu/\rmd \gamma})^2 \rmd \gamma$ is extended to non-probability measures in the natural way, with any dominating measure $\gamma$.} and for every $P\in \calP(\delta)$, there exists $i\in [N]$ such that $P_i^{\rL} \le P \le P_i^{\rU}$ (with $\le$ understood in a pointwise manner). 
The following result states a general entropic upper bound on the Hellinger distance $H(\widehat{P}, \overline{P})$.

\begin{lem}\label{lemma:entropic-upper-bound}
There exists a universal constant $C>0$ such that the following holds. Let
\begin{align*}
\Psi(\delta) \ge \int_{\delta^2/C}^{\delta} \sqrt{\log \bracketing(u,\calP(\overline{P}, \delta),H)}\rmd u \vee \delta
\end{align*}
be any function such that $\Psi(\delta)/\delta^2$ is non-increasing in $\delta$, and $\delta_n>0$ satisfy $\sqrt{n}\delta_n^2 \ge C \Psi(\delta_n) $. Then for all $\delta\ge \delta_n$, any $\beta$-approximate MLE $\widehat{P}$ satisfies
\begin{align*}
\bP\pth{ H^2(\widehat{P}, \overline{P}) \ge \delta^2 + \frac{C}{n}\log\frac{1}{\beta} } \le C\exp\bpth{-\frac{n\delta^2}{C^2}}. 
\end{align*}
\end{lem}

In the i.i.d. case where $P_i\equiv P$, \Cref{lemma:entropic-upper-bound} reduces to the classical upper bounds on $H^2(\widehat{P},P)$ in \citep{wong1995probability,geer2000empirical}. The same proof technique extends to the i.n.i.d. setting, where the MLE $\widehat{P}$ is instead close to the average distribution $\overline{P}$. In heteroskedastic mean estimation, this average distribution is $\overline{P} = \frac{1}{n}\sum_{i=1}^n \calN(\mu,\sigma_i^2) = f_{\mu,G_n}$, which motivates the form of \Cref{lemma:density_estimation}. For completeness, we include the proof of \Cref{lemma:entropic-upper-bound} in the appendix, as the i.n.i.d. case in \citep[Chapter 8.3]{geer2000empirical} does not cover this precise setting. 
% For completeness, 

\subsection{Improved Covering Results for the Normal Scale Mixture}

The central metric entropy bound in this section is the following: 
\begin{thm}\label{thm:metric-entropy}
Let $\calP = \{f_{\mu, G}: \mu\in \mathbb{R}, \supp(G)\subseteq [\sigmin,\sigmax] \}$ such that $\overline{P}\in \calP$. There exists a universal constant $C>0$ such that for $0<\varepsilon\le \frac18$, 
\begin{align*}
\log \bracketing(\varepsilon, \calP(\overline{P}, \frac{1}{8}), H) \le C\log^6\bpth{\frac{\sigmax}{\varepsilon \sigmin}}\log\log\bpth{\frac{\sigmax}{\varepsilon \sigmin}}. 
\end{align*}
\end{thm}
% \YH{TODO: find an explicit $c$, and adjust theorem numbering.}

In conjunction with \Cref{lemma:entropic-upper-bound}, \Cref{thm:metric-entropy} gives the target upper bound of \Cref{lemma:density_estimation}. The main technical challenge in the proof of \Cref{thm:metric-entropy} is to derive the bracketing entropy bound for the class of \emph{normal scale mixtures}, i.e., the subset of $\calP$ with $\mu=0$ and a general mixture distribution $G$ on $[\sigmin, \sigmax]$. Specifically, using a reparametrization $t=1/\sigma$ and standard truncation arguments (deferred to \Cref{append:metric-entropy}), \Cref{thm:metric-entropy} follows from the following technical lemma. 

\begin{lem}\label{lemma:key-covering}
Let $0<t_{\min}\le 1\le t_{\max}$ and $0<x_{\min}\le 1\le x_{\max}$. Let $\calF$ be the class of all functions $\bE_{t\sim H}[t\exp(-t^2x^2/2)]$ with $H\in \calP([t_{\min},t_{\max}])$, and $L_\infty([x_{\min},x_{\max}])$ denotes the sup norm on the interval $[x_{\min},x_{\max}]$. Then for $\varepsilon\in (0,1/2)$, 
\begin{align*}
\log N(\varepsilon, \calF, L_\infty([x_{\min},x_{\max}])) \le C\log^6\pth{\frac{t_{\max}x_{\max}}{\varepsilon t_{\min}x_{\min}}}\log\log\pth{\frac{t_{\max}x_{\max}}{\varepsilon t_{\min}x_{\min}}}. 
\end{align*}
\end{lem}

In \Cref{lemma:key-covering}, the $\text{polylog}(\frac{1}{\varepsilon})$ dependence on $\varepsilon$ is perhaps unsurprising as $\calF$ is effectively a parametric family by discretizing the support of $H$; the key challenge is to obtain a logarithmic dependence on $\frac{t_{\max}}{t_{\min}}$ and $\frac{x_{\max}}{x_{\min}}$. This is precisely where existing work based on polynomial approximation or Taylor series fail to capture, as detailed in the following sections.

\subsubsection{Generalized moment matching}

In the literature for normal location or scale mixtures, the prevalent approach for establishing metric entropy bounds is through the idea of \emph{moment matching}. Specifically, it is shown that if $H$ and $H'$ have the same first $L$ moments, then the difference $\| \bE_{t\sim H}[t\exp(-t^2x^2/2)] - \bE_{t\sim H'}[t\exp(-t^2x^2/2)] \|_\infty$ decays exponentially in $L$. Since Carath\'{e}odory's theorem shows that matching first $L$ moments can be realized by a discrete distribution with at most $O(L)$ atoms, covering such discrete distributions typically leads to a metric entropy bound of $\widetilde{O}(L)$. This idea has been used for approximating normal scale mixtures in \citep{ghosal2001entropies, ghosal2007posterior,ignatiadis2025empirical}; an improved bound is also obtained via \emph{local moment matching} in \citep{jiang2009general} for normal location mixtures, with extensions to Poisson mixtures in \citep{shen2022empirical}. 

However, such moment matching approaches (even the localized version) in normal scale mixtures suffer from a polynomial dependence on the ratio $\frac{t_{\max}}{t_{\min}}$. The underlying mathematical reason is that, such moment matching arguments rely on the following polynomial approximation
\begin{align*}
\sup_{x\in [x_{\min},x_{\max}], t\in [t_{\min},t_{\max}]} \Big| te^{-\frac{t^2x^2}{2}} - \sum_{k=1}^L g_k(x) t^k \Big| \le \varepsilon, 
\end{align*}
and this polynomial is usually chosen to be a truncation of Taylor series. Based on this approximation, it is clear that $\bE_{t\sim H}[t\exp(-t^2x^2/2)]$ approximately depends only on $\bE_{t\sim H}[t^k], k=1,\dots,L$. However, even for a fixed $x$, approximation theory tells that the degree $L$ must have a polynomial dependence on $(t_{\min}, t_{\max})$. For example, \citep{aggarwal2022optimal} shows that to approximate $e^{-t}$ with a uniform approximation error $\varepsilon$ on $t\in [0,B]$, the degree of the polynomial must be at least $\widetilde{\Omega}(\sqrt{B})$, exhibiting a polynomial dependence on $B$. 

To address this issue, we make a simple yet important observation that we may apply a generalized moment matching with suitably chosen basis functions: 
\begin{align}\label{eq:moment-matching-general}
\sup_{x\in [x_{\min},x_{\max}], t\in [t_{\min},t_{\max}]} \Big| t e^{-\frac{t^2x^2}{2}} - \sum_{k=1}^L a_k(t) g_k(x) \Big| \le \varepsilon. 
\end{align}
Here, by choosing $a_k(t)$ beyond polynomials, we may find a better approximating basis of $t\exp(-t^2x^2/2)$ to achieve a smaller approximation error. The following lemma is an easy metric entropy bound if we can construct bounded functions $\{a_k(t), g_k(x)\}_{k=1}^L$ such that \eqref{eq:moment-matching-general} holds. 

\begin{lem}\label{lemma:generalized-moment-matching}
Suppose \eqref{eq:moment-matching-general} holds with $|a_k(t)| + t|a_k'(t)|\le A$ for all $t\in [t_{\min},t_{\max}]$ and $|g_k(x)|\le G$ for all $x\in [x_{\min}, x_{\max}]$. Then $\log N(2\varepsilon, \calF, L_\infty([x_{\min},x_{\max}])) \le CL\log\frac{AGL\log(t_{\max}/t_{\min})}{\varepsilon}$. 
\end{lem}

To choose the basis, we use the natural idea of writing $u=\log x+ \log t$ and try to approximate $\exp(-t^2x^2/2) = \exp(-\frac{1}{2}e^{2u}) \approx P_L(u)$, where $P_L$ is a polynomial of degree at most $L$. Expanding $P_L(u)$ into a bivariate polynomial in $(\log x, \log t)$, we obtain the basis functions $(a_k(t), g_k(x))$ in \eqref{eq:moment-matching-general} in the form of $t(\log t)^i$ and $(\log x)^j$. Therefore, it remains to find a uniform polynomial approximation for $\exp(-\frac{1}{2}e^{2u})$ on $u\in [\log(x_{\min}t_{\min}), \log(x_{\max}t_{\max})]$, or equivalently, for $\exp(-Ke^{\lambda u})$ on $u\in [-1,1]$ after translation and scaling, with $K=\frac{1}{2}x_{\min}x_{\max}t_{\min}t_{\max}>0$ and $\lambda = \log\frac{x_{\max}t_{\max}}{x_{\min}t_{\min}}$.

\subsubsection{Chebyshev approximation}
In this section we solve the polynomial approximation problem of approximating $\exp(-Ke^{\lambda u})$ on $[-1,1]$. As commonly used in the moment matching literature, a first natural idea is to apply the Taylor approximating polynomial of $\exp(-Ke^{\lambda u})$ at $u=0$. However, this leads to a poor approximation: after some algebra, we can show that the $n$-th Taylor coefficient is $a_n = \frac{\lambda^n}{n!}e^{-K}B_n(-K)$, where $B_n(x)=\sum_{k=0}^n S(n,k)x^k$ is the Bell/Touchard polynomial and $S(n,k)$ is the Stirling number of the second kind. Since $|B_n(-K)|$ grows at a speed faster than $K^n$, we see that $|a_n|\to 0$ only if $n\gg \lambda K$. However, since $K$ is polynomial in $(x_{\min}, x_{\max}, t_{\min}, t_{\max})$, this means that the degree of the Taylor approximating polynomial must also be a polynomial in $(x_{\min}, x_{\max}, t_{\min}, t_{\max})$. 
To improve over this dependence, we critically make use of the \emph{Chebyshev approximation} to achieve an approximation error \emph{independent of $K$}, as summarized in the following lemma. 

\begin{lem}\label{lem:chebyshev_approximation}
Let $h(x)=\exp(-K e^{\lambda x})$ on $[-1,1]$, with $K>0$ and $\lambda>0$.
Let $P_L$ be the degree-$L$ Chebyshev polynomial of $h$. Then
$\|h-P_L\|_\infty\le \varepsilon$ for $
L = O(\log(\frac{1}{\varepsilon})+\lambda \log(\frac{\lambda}{\varepsilon})). $
% Furthermore, the coefficients of $P_r$ are bounded in magnitude by $2^{r+2}$. 
\end{lem}

\begin{proof}
The proof bounds the Chebyshev approximation error using the Bernstein ellipse theorem.

\begin{lem}[Bernstein Ellipse Theorem, {\citep[Theorem 8.2]{trefethen2019approximation}}]
\label{lem:bernstein}
Let $h$ be analytic in the open Bernstein ellipse
$E_\rho=\{z\in\bC: z=\tfrac{1}{2}(\rho e^{i\theta}+\rho^{-1}e^{-i\theta}),
\theta\in[0,2\pi]\}$ for some $\rho>1$, and suppose $|h(z)|\le M$ on $E_\rho$.
Let $P_L$ be the degree-$L$ Chebyshev truncation of $h$ on $[-1,1]$. Then $
\|h-P_L\|_\infty \le \frac{2M}{\rho-1}\rho^{-L}.$
\end{lem}

To conclude from \Cref{lem:bernstein}, note that the Bernstein ellipse $E_\rho$ is contained in the strip
$|\Im z|\le \frac{1}{2}(\rho-\rho^{-1})\le \rho-1$ for $\rho \in [1,2]$. In addition, $|h(z)| = \exp(-Ke^{\lambda x}\cos(\lambda y))$ for $z=x+iy$. Therefore, for $\rho = \min\{2, 1 + \frac{\pi}{2\lambda}\}$, we have $\cos(\lambda y) \ge 0$ for $|y|\le \rho -1$ and thus $|h(z)|\le 1$ on $E_\rho$. Now \Cref{lem:chebyshev_approximation} directly follows from \Cref{lem:bernstein}. 
\end{proof}

Importantly, the final polynomial degree $L$ in \Cref{lem:chebyshev_approximation} is only polynomial in $\lambda$, and thus poly-logarithmic in $\frac{x_{\max}t_{\max}}{x_{\min}t_{\min}}$; this is the target feature of \Cref{lemma:key-covering}. 
It is interesting to note why Chebyshev approximation performs better than Taylor approximation here. In fact, Taylor approximation gives a convergence rate $\rho^{-n}$ when $|h|$ is bounded on the disc $|z|=\rho$, whereas Chebyshev approximation operates on a different geometry $E_\rho$ which more tightly envelops the desired interval. 
Since $|h(z)|$ could grow rapidly along the imaginary axis, here the Bernstein ellipse is much more suitable than the disc.

%%%%%%%%%%%%%%%%%%%%%%%%%%%%%%%%%%%%%%%%%%%%%%%%%%%%%%%%%%%%%%%%%
\section{Numerical experiments and algorithms}\label{sec:numerics}
%%%%%%%%%%%%%%%%%%%%%%%%%%%%%%%%%%%%%%%%%%%%%%%%%%%%%%%%%%%%%%%%%

In this section, we summarize important properties of the NPMLE that are well-known in the empirical Bayes literature \citep{lindsay1983geometry,lindsay1995mixture}, specialized to normal scale mixtures. We also detail our computational approach for finding the NPMLE based on the Frank--Wolfe procedure adapted from \citep{han2025besting}, and numerically compare our approach to previous methods in the literature such as median and iterative truncation \citep{liang2020learning, pensia2022estimating}. 
\vspace{-0.5em}
\subsection{Basic properties of the NPMLE}
The main difficulty in solving the MLE \eqref{eq:EB_MLE} is to solve for the NPMLE $\widehat{G}$ once $\widehat{\mu}$ is fixed. This section describes some basic properties of $\widehat{G}$, and WLOG we assume that $\widehat{\mu}=0$. Since the log-likelihood is concave in $G$, it is well-known \citep{lindsay1983geometry} that the KKT condition for $\widehat{G}$ is
\begin{align}\label{eq:KKT}
D_{\widehat{G}}(\sigma) := D_{\widehat{G}}(\sigma; X^n) := \frac{1}{n}\sum_{i=1}^n \frac{\frac{1}{\sigma}\varphi(\frac{X_i}{\sigma})}{f_{0,\widehat{G}}(X_i)} \leq 1, \quad \forall \sigma \ge 0,
\end{align}
and equality holds iff $\sigma \in \supp(\widehat{G})$. Based on this, we have the following structural result. 
% for $\widehat{G}$. 
% Characterizing optimality of $(\widehat{\mu},\widehat{G})$ in terms of $D_{\widehat{G}}(\sigma;\widehat{\mu})$ gives rise to key structural properties of the estimates $\widehat{G}^{\text{EB}}$. In particular, the following lemma shows that $\widehat{G}$ is a discrete measure supported on a finite set of atoms in the interval $[\overline{X}_{\text{min}}, \overline{X}_{\text{max}}]$, an important structural property of the NPMLE which is typical in the empirical Bayes literature for heteroskedastic Gaussian mixtures \citep{soloff2025multivariate} and makes the NPMLE computationally tractable in low dimensions \citep{koenker2014convex}.
\begin{lem}\label{lemma: structural}
Suppose $X_i\neq 0$ for all $i\in [n]$. Then the NPMLE $\widehat{G}$ is a discrete distribution with $|\supp(\widehat{G})|\le |\{|X_1|,\dots,|X_n|\}|$, $\supp(\widehat{G})\subseteq [\min_{i\in [n]}|X_i|, \max_{i\in [n]}|X_i|]$, and unique. 
\end{lem}

\subsection{Successive maximization algorithm for the joint MLE}\label{sec: successive max}

Motivated by the NPMLE properties in \Cref{lemma: structural}, especially the discrete nature of $\widehat{G}$, we use a fully-corrective Frank--Wolfe algorithm \citep{frank1956algorithm} (also known as the vertex-direction method in the mixture model literature) to compute $\widehat{G}$ with fixed $\widehat{\mu}$, in a similar manner to \citep{han2025besting}. Specifically, given an initialization $\widehat{G}_0$, for $t=1,2,\dots$, we repeat the following steps until convergence: (1) find a new atom $\sigma_{t} \in \argmax_{\sigma > 0} D_{\widehat{G}_{t-1}}(\sigma; X^n - \widehat{\mu})$ via a grid search,
and then (2) add $\sigma_t$ to the support of $\widehat{G}$ and optimize over the weights via a convex program: 
\begin{align*}
    \widehat{G}_t = \argmax_{\supp(G) \subseteq \supp(\widehat{G}_{t-1}) \cup \{\sigma_{t}\}} \frac{1}{n}\sum_{i=1}^n \log f_{\widehat{\mu},G}(X_i).
\end{align*}
We note that in our implementation we do not enforce the constraints $\supp(\widehat{G})\subseteq [\sigmin, \sigmax]$; this is because by \Cref{lemma: structural}, the data $X^n$ have already constrained the support of $\widehat{G}$. 
Based on the Frank--Wolfe algorithm for fixed $\widehat{\mu}$, our final algorithm for joint optimization of $(\widehat{\mu}, \widehat{G})$ proceeds via successive maximization. Given an initialization $\widehat{\mu}_0$ (chosen as the median in our implementation), for $t = 1,2,\dots$, we iterate the following steps until convergence: (1) run the Frank--Wolfe algorithm with $\widehat{\mu}$ fixed at $\widehat{\mu}_{t-1}$ to obtain $\widehat{G}_t$; (2) perform a grid search to find $\widehat{\mu}_t$ that maximizes the total likelihood for fixed $\widehat{G}_t$. Since the likelihood improves at each iteration, we terminate when the increase in likelihood falls below a prescribed threshold. Although the overall algorithm is susceptible to local maxima, our experiments show that it converges well in practice. 
\subsection{Experiments}
\begin{figure}[ht]
    \centering
    \subfigure[$m=n^{1/3}$]{
    \includegraphics[width=0.48\linewidth]{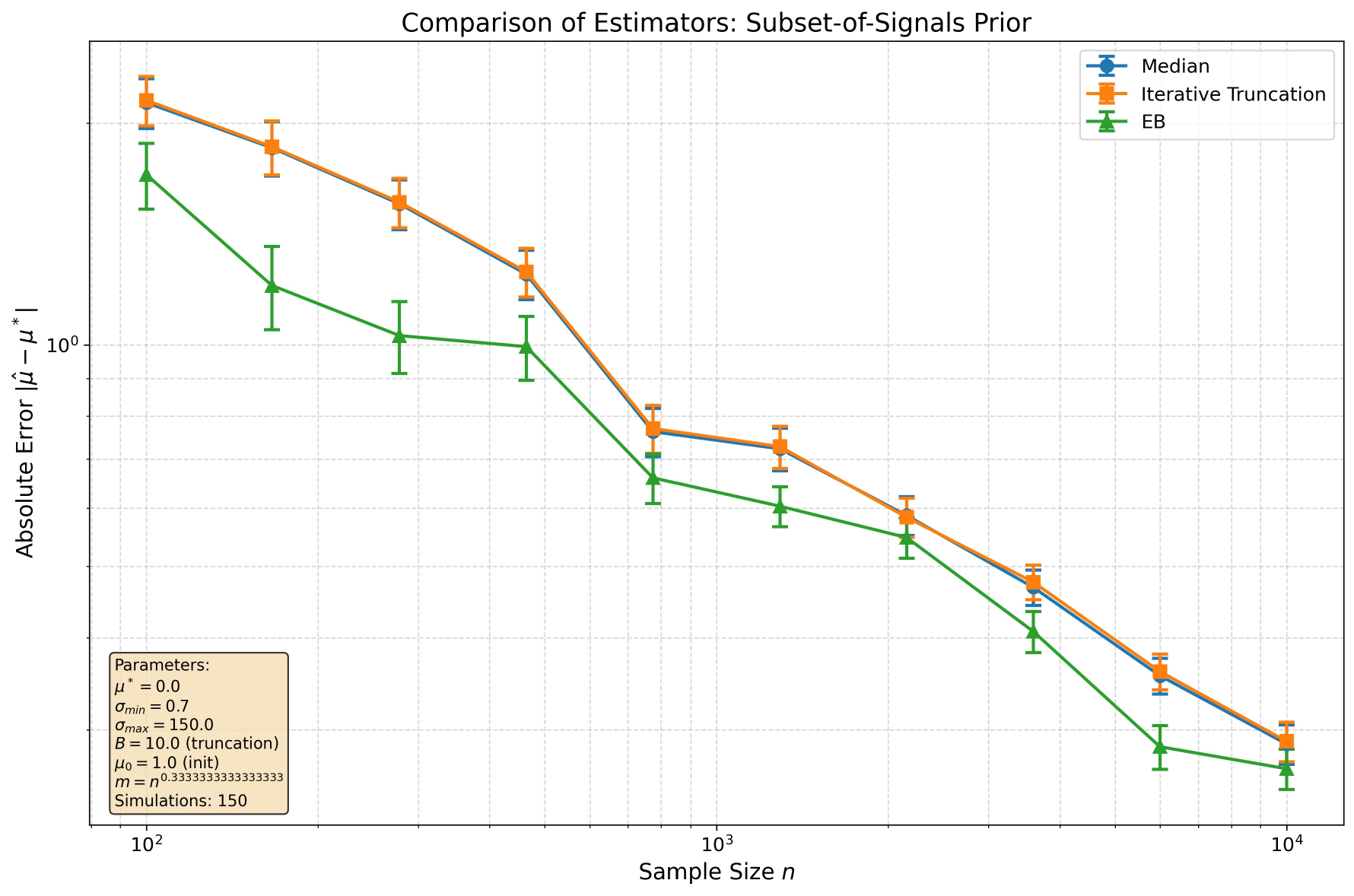}}
    % \hfill
    \subfigure[$m=\sqrt{n}$]{\includegraphics[width=0.48\linewidth]{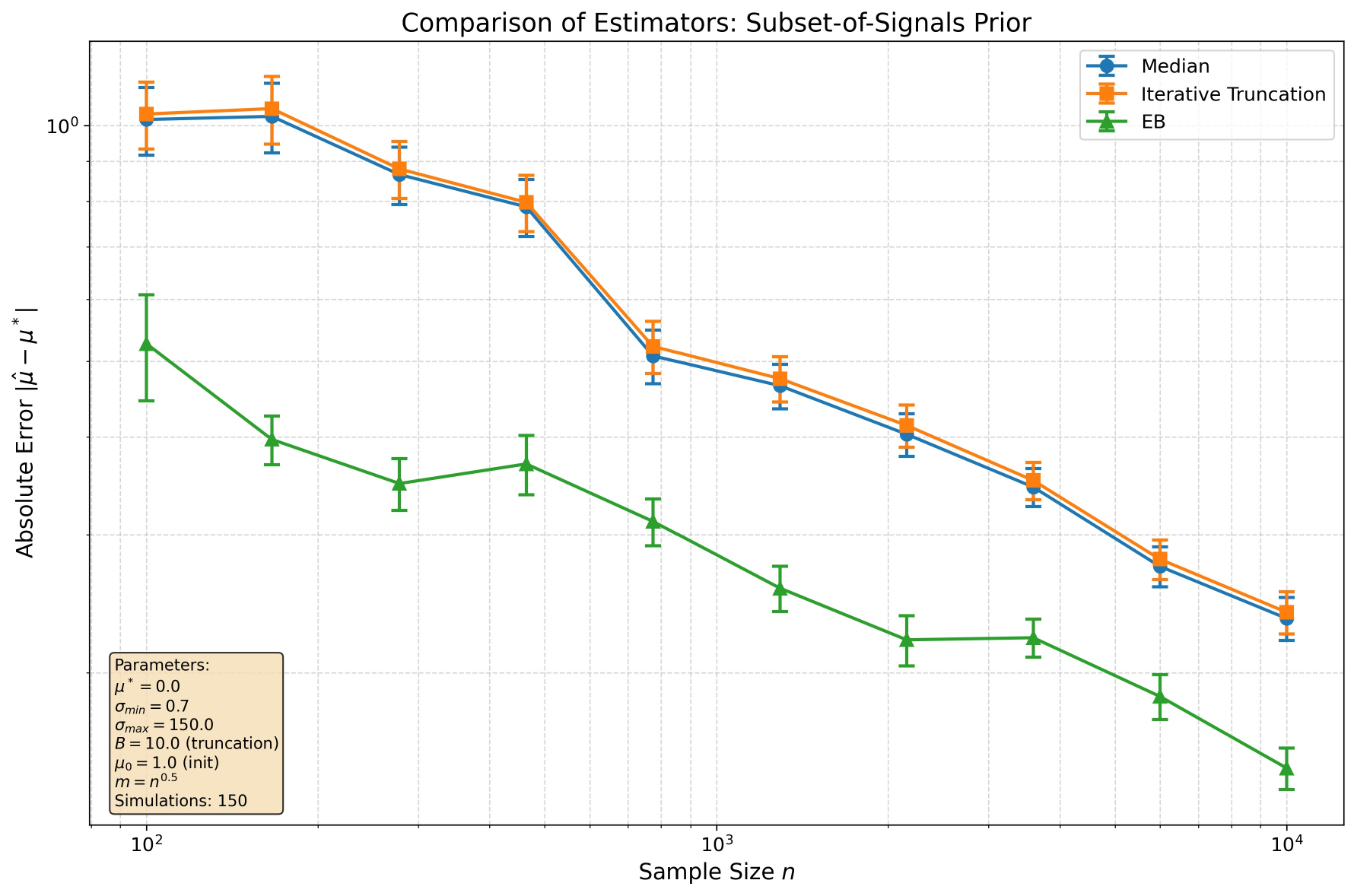}}
    \caption{Average absolute estimation errors for $\widehat{\mu}^{\EB}$, sample median, and iterative truncation over $150$ simulations, under the subset-of-signals prior $G_n=\frac{m}{n}\Unif([0.7,1]) + \frac{n-m}{n}\Unif([1,150])$, with different choices of $(m,n)$.}%Left: Simulation comparing the median, iterative truncation with initialization $(\mu_0, B) = (1, 10)$, and NPMLE for the subset-of-signals prior with parameters $\mu = 2$, $\sigmin = 0.8$, $\sigmax = 150$, and $m = n^{1/3}$. Right: Simulation comparing the Oracle MLE $\widehat{\mu}_{\text{known-}\sigma}$ (see \eqref{eq: mle known si}) to the NPMLE $\widehat{\mu}^{\text{EB}},$ median, and MLE with known prior $G$. The prior $G$ is the two-point prior $\frac{1}{100}\delta_1 + \frac{99}{100}\delta_{30}$. Both use error bars up to one standard deviation for $N = 100$ simulations.}
    \label{fig:sos_prior}
\end{figure}

\begin{figure}[ht]
    \centering
    \subfigure[$G_n = \frac{1}{100}\delta_1 + \frac{99}{100}\delta_{30}$]{
    \includegraphics[width=0.48\linewidth]{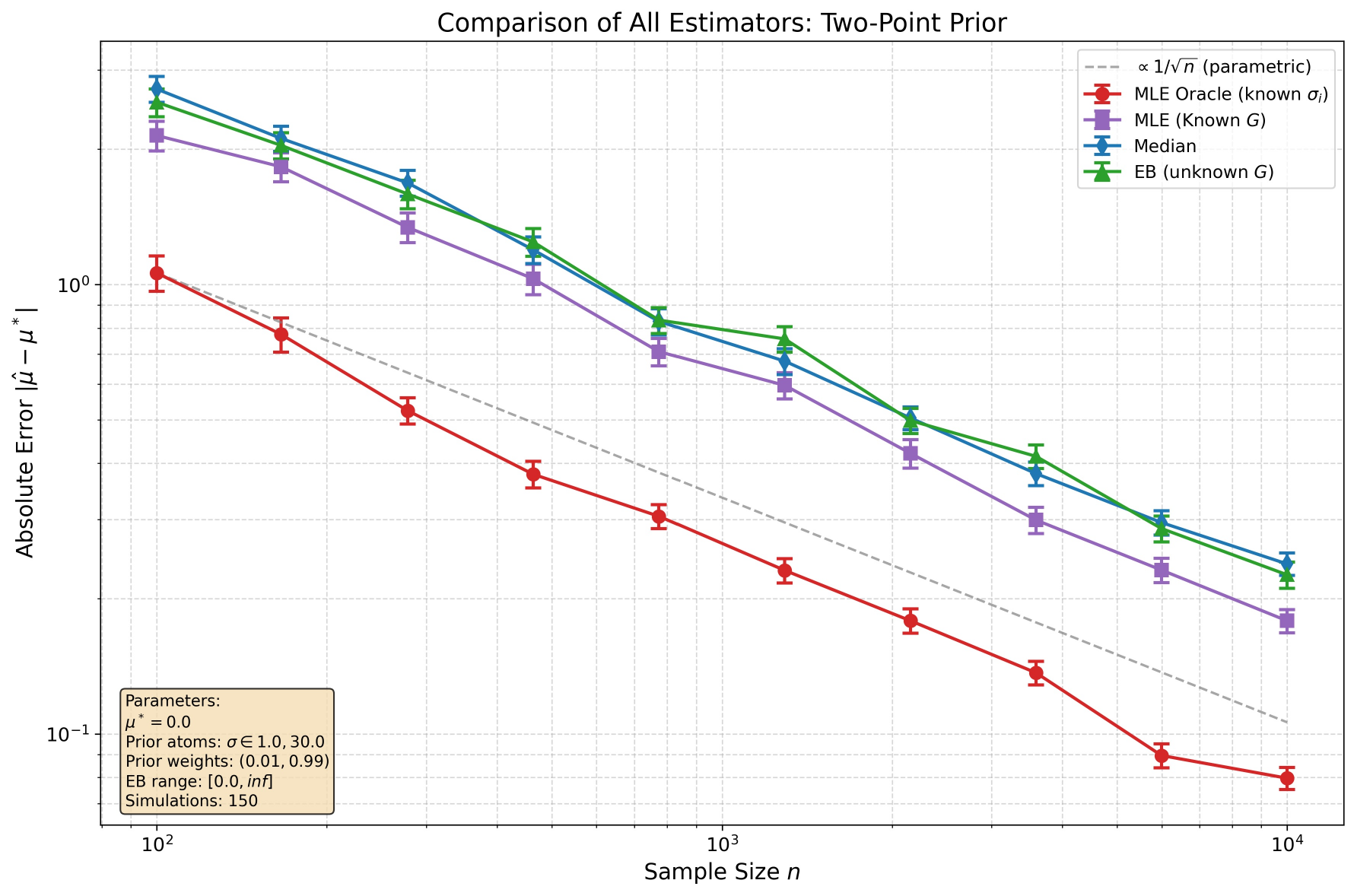}}
    % \hfill
    \subfigure[$G_n = \frac{1}{20}\delta_{1} + \frac{9}{10}\delta_{60}+\frac{1}{20}\delta_{100}$]{\includegraphics[width=0.48\linewidth]{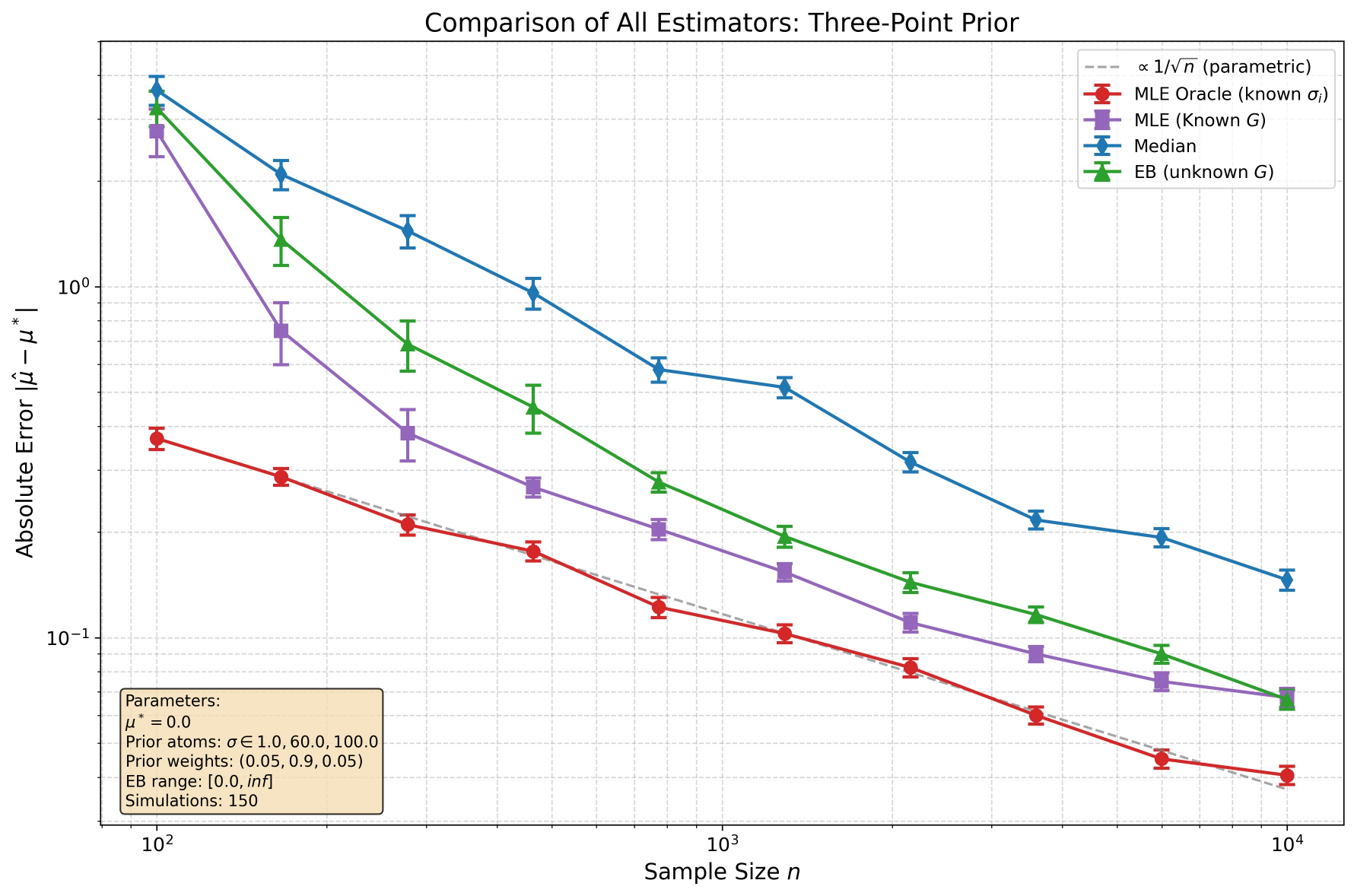}}
    \caption{Average absolute estimation errors for $\widehat{\mu}^{\EB}$, sample median, and two oracle MLEs over $150$ simulations, under two-point and three-point scale mixtures priors.}
    \label{fig:finite-point}
\end{figure}

In this section, we illustrate the empirical  performance of our estimator $\widehat{\mu}^{\EB}$ on a variety of empirical distributions (priors) $G_n$ of $(\sigma_1,\dots,\sigma_n)$, including the subset-of-signals prior $G_n=\frac{m}{n}\Unif([\sigmin,1])+\frac{n-m}{n}\Unif([1,\sigmax])$, two-point and three-point scale mixture priors. We also consider the equal and quadratic variance models, but defer these results to \Cref{appendix: extra simul} for space considerations. 
%We detail two sets of experiments, saving additional experiments for \Cref{appendix: extra simul}. 
On the subset-of-signals prior, we compare our estimator with two other estimators designed for the subset-of-signals problem: the sample median and the iterative truncation estimator in \citep{liang2020learning}. On the two-point and three-point priors, in addition to the median estimator, we compare with two \emph{oracle} estimators: the linear estimator $\widehat{\mu}_{\text{known}-\sigma}$ in \eqref{eq: mle known si} which we call the oracle MLE, and the MLE of $\mu$ in \eqref{eq:EB_MLE} with the true knowledge of $G_n$. We emphasize that apart from the grid search, number of starting atoms, and stopping criterion, our algorithm is completely free of tuning parameters.

\Cref{fig:sos_prior} and \ref{fig:finite-point} display at the log scale the average absolute estimation error of $\mu$ over $150$ simulations for both sets of experiments over a range of scales for $n$. In \Cref{fig:sos_prior}, we set %$(\sigmin,\sigmax)= (0.7,150)$ and 
$m\in \{ n^{1/3}, n^{1/2}\}$, and the hyperparameters $(B,\mu_0)=(10,1)$ for the iterative truncation estimator. We observe that $\widehat{\mu}^{\EB}$ outperforms the other two estimators in both settings. In \Cref{fig:finite-point}, the scale mixture is a two-point prior $G_n = \frac{1}{100}\delta_1 + \frac{99}{100}\delta_{30}$ or three-point prior $G_n = \frac{1}{20}\delta_{1} + \frac{9}{10}\delta_{60}+\frac{1}{20}\delta_{100}$. We observe that the performance of $\widehat{\mu}^{\EB}$, obtained without knowledge of $G_n$, is close to that of the MLE computed with knowledge of $G_n$. This indicates that the proposed successive maximization algorithm, in particular the Frank--Wolfe procedure used to estimate $\widehat{G}$, exhibits good empirical convergence, despite the potential for convergence to local maxima. % We defer the experimental results under more settings to the appendix.  

\section{Conclusion and Open Problems}\label{sec:conclusion}
%%%%%%%%%%%%%%%%%%%%%%%%%%%%%%%%%%%%%%%%%%%%%%%%%%%%%%%%%%%%%%%%%

In this work, we demonstrate that a simple empirical Bayes approach based on maximum likelihood estimation achieves near-optimal performance for heteroskedastic mean estimation across a wide range of instances. Compared with traditional empirical Bayes approaches, our procedure replaces heterogeneity with a mixing distribution (prior) and learns it from the data, but does not invoke explicit posterior inference. 
Instead, the learned prior is treated as a nuisance estimate and is used in constructing our final mean estimator, which takes the form of a profile MLE. 
It is thus an interesting finding that this empirical Bayes approach still achieves near-optimal statistical performance in heteroskedastic mean estimation. Several interesting questions remain open. 

% Our main contribution is showing that the joint MLE of the mean parameter $\mu$ and the nonparametric mixing distribution $G$ adaptively attains minimax rates for the subset-of-signals problem, including the correct phase transition at $m = n^{1/4}$, without requiring knowledge of the signal size $m$ or any tuning parameters beyond the minimum variance $\sigma_{\min}$. The key technical innovations enabling this result are improved metric entropy bounds for normal scale mixtures via Chebyshev polynomial approximations, which yield polylogarithmic rather than polynomial dependence on $1/\sigma_{\min}$, and a new functional inequality that bounds the Hellinger modulus of continuity for location families under structural assumptions on the mixing distribution.

\paragraph{Efficient computation.} From a computational perspective, the optimization problem in \eqref{eq:EB_MLE} is nonconvex (in $\mu$) and infinite-dimensional (in $G$). While our numerical experiments suggest that successive maximization with Frank--Wolfe performs well in practice, it would be valuable to develop provably efficient algorithms and/or to understand the landscape of this optimization problem. 

\paragraph{Hyperparameters $(\sigmin, \sigmax)$.} Our theoretical guarantee involves a logarithmic factor in $\frac{\sigmax}{\sigmin}$, which primarily arises from our proof strategy for establishing the density estimation guarantee in \Cref{lemma:density_estimation}. Indeed, when $\sigmin = 0$, the likelihood in \eqref{eq:EB_MLE} can be made unbounded by assigning positive mass to $\widehat{G}(\{0\})$ and choosing $\widehat{\mu} \in \sth{X_1,\dots,X_n}$, in which case $H^2(f_{\widehat{\mu},\widehat{G}}, f_{\mu,G})$ is no longer small. However, this does \emph{not} imply that the resulting estimator $\widehat{\mu}$ fails to estimate $\mu$ accurately. For example, in our numerical experiments, we always set $\sigmin = 0$ and $\sigmax = \infty$, yet the resulting estimator continues to perform well. As another illustration, the Le Cam--Birg\'{e}-type mean estimator studied in \citep{compton2025attainability} likewise satisfies $\widehat{\mu} \in \sth{X_1,\dots,X_n}$, suggesting that estimators of this form can still achieve good performance. It is therefore an interesting question to remove the assumptions on $(\sigmin, \sigmax)$.
\paragraph{Le Cam's lower bound for compound problems.} Although $\omega_{H^2,G_n}(\frac{1}{n})$ is a minimax lower bound for the mean estimation problem where $X_1,\dots,X_n$ are i.i.d. drawn from $f_{\mu,G_n}$ with known $G_n$, the same claim is not rigorous in the compound setting. In the compound setting, Le Cam's method only gives the following lower bound: let $
\bP_{\mu, G_n} := \bE_{\pi\sim \Unif(S_n)} [ \bigotimes_{i=1}^n \calN(\mu, \sigma_{\pi(i)}^2) ] $
be the $n$-dimensional ``permutation mixture'', then a minimax lower bound for mean estimation is
\begin{align*}
\sup\bsth{ |\mu_1-\mu_2|: \mu_1, \mu_2 \in \bR,  H^2(\bP_{\mu_1,G_n}, \bP_{\mu_2,G_n}) \le 2-\Omega(1)}. 
\end{align*}
Assuming a mean-field approximation $\bP_{\mu,G_n} \approx f_{\mu,G_n}^{\otimes n}$ for permutation mixtures, the above quantity would essentially reduce to $\omega_{H^2,G_n}(\frac{1}{n})$. However, although recent work \citep{han2024approximate,liang2025sharp} shows that such a mean-field approximation holds quantitatively and independently of the dimension $n$, these results only guarantee $\TV(\bP_{\mu,G_n}, f_{\mu,G_n}^{\otimes n}) \le 1 - c$ for a small constant $c > 0$. Consequently, a direct application of the triangle inequality yields a vacuous bound (incurring a cost of $2 \cdot \TV$) and is therefore insufficient to establish the desired equivalence.

% \vspace{-0.4cm}
\paragraph{Multivariate settings.} Extending our framework to multivariate settings, where observations $X_i \in \mathbb{R}^d$ have heteroskedastic covariance matrices, presents both statistical and computational challenges. On the statistical side, the density estimation-based proof strategy becomes unsuitable, since the Hellinger error typically scales exponentially in $d$. On the computational side, computing the NPMLE in high dimensions is challenging, and standard algorithms such as Frank–Wolfe become less effective.

% Acknowledgments---Will not appear in anonymized version
\acks{ Yanjun Han would like to thank Cun-Hui Zhang for helpful discussions at an early stage of this project. Abhishek Shetty would like to thank Shyam Narayanan, Ilias Diakonikolas, and Daniel Kane for helpful discussions. The numerical experiments were supported by NYU High Performance Computing (HPC) resources. We thank the reviewers for helpful comments.}

\bibliography{refs}

\appendix

%\input{app_related_work.tex}

%\section{Discussion}\label{sec:discussion}
%\input{app_discussion.tex}

% \section{Proof of \Cref{lemma:density_estimation}}

% Placeholder - full proof in appendix
% The proof follows from the entropic upper bound combined with the covering number bound.

\section{Proof of \Cref{thm:general}}\label{sec:symmetrization}
% \subsection{Proof of Lemma \ref{lemma:symmetrization}} \label{sec:symmetrization}

% \paragraph{Relating Error in Density to Error in Mean Estimation} \YH{will move to appendix.}

To complete the proof of \Cref{thm:general}, we prove a symmetrization inequality which reduces possibly different scale mixtures to a single mixture. 
\begin{lem}\label{lemma:symmetrization}
Let $\mu_1, \mu_2 \in \bR$, and $G_1, G_2$ be two prior distributions over $[0,\infty)$. Then
\begin{align*}
H^2(f_{\mu_1,G_1}, f_{\mu_2,G_2}) \ge \frac{1}{4} H^2(f_{|\mu_1-\mu_2|,G_1}, f_{-|\mu_1-\mu_2|,G_1}). 
\end{align*}
\end{lem}

% \begin{proof}[Proof of \cref{thm:general}]
    The proof of \cref{thm:general} follows by $(\mu_1,G_1) = (\mu,G_n)$ and $(\mu_2,G_2) = (\widehat{\mu},\widehat{G})$ in \Cref{lemma:symmetrization} and \Cref{lemma:density_estimation}. 
% \end{proof}
Crucially, thanks to \Cref{lemma:symmetrization}, the prior $G_2$, which corresponds to the estimated prior $\widehat{G}$, does not appear on the right-hand side and is therefore not subject to any structural requirements such as those imposed on $G_1$. 
This asymmetry is the key technical reason why our estimator adapts to the signal size $m$ in the subset-of-signals problem.

\begin{proof}[Proof of \Cref{lemma:symmetrization}]
By translation and reflection invariance, we may assume that $\mu_1 = \mu \ge 0$, and $\mu_2 = 0$. The proof will rely on an easy symmetric relationship: 
\begin{align}\label{eq:reflection}
f_{\mu, G}(x) = \bE_{\sigma\sim G}\qth{\frac{1}{\sigma}\varphi\pth{\frac{x-\mu}{\sigma}}} = \bE_{\sigma\sim G}\qth{\frac{1}{\sigma}\varphi\pth{\frac{-x+\mu}{\sigma}}} = f_{-\mu,G}(-x). 
\end{align}
We repeatedly apply this symmetry to get
\begin{align*}
H^2(f_{\mu,G_1}, f_{0,G_2}) &= \int_{-\infty}^\infty \bpth{ \sqrt{f_{\mu,G_1}(x)} - \sqrt{f_{0,G_2}(x)} }^2 \rmd x \\
&= \frac{1}{2} \int_{-\infty}^\infty \bqth{\bpth{ \sqrt{f_{\mu,G_1}(x)} - \sqrt{f_{0,G_2}(x)} }^2 + \bpth{ \sqrt{f_{\mu,G_1}(-x)} - \sqrt{f_{0,G_2}(-x)} }^2} \rmd x \\
&\overset{\prettyref{eq:reflection}}{=} \frac{1}{2} \int_{-\infty}^\infty \bqth{\bpth{ \sqrt{f_{\mu,G_1}(x)} - \sqrt{f_{0,G_2}(x)} }^2 + \bpth{ \sqrt{f_{\mu,G_1}(-x)} - \sqrt{f_{0,G_2}(x)} }^2} \rmd x \\
&\ge \frac{1}{4}\int_{-\infty}^\infty \bpth{ \sqrt{f_{\mu,G_1}(x)} - \sqrt{f_{\mu,G_1}(-x)} }^2 \rmd x \\
&\overset{\prettyref{eq:reflection}}{=}  \frac{1}{4}\int_{-\infty}^\infty \bpth{ \sqrt{f_{\mu,G_1}(x)} - \sqrt{f_{-\mu,G_1}(x)} }^2 \rmd x \\
&= \frac{1}{4}H^2(f_{\mu,G_1}, f_{-\mu,G_1}), 
\end{align*}
and the inequality step uses $a^2+b^2\ge \frac{(a-b)^2}{2}$. This is the desired result. 
\end{proof}

\section{Density Estimation: Proof of \Cref{lemma:density_estimation}}\label{app:proof_density_estimation}

In this section, we provide the proofs of the key lemmas used to establish \Cref{lemma:density_estimation}. 

\subsection{Proof of  \cref{lemma:entropic-upper-bound}} \label{app:proof_entropic_upper_bound}

Our proof mostly follows the arguments in \citep{geer2000empirical}. We first mimic the arguments of \citep[Lemma 4.1]{geer2000empirical}. Let $P_n$ be the empirical distribution of $X_1,\dots,X_n$, then
\begin{align*}
-\frac{1}{n}\log\frac{1}{\beta} &\stepa{\le} \int \log\frac{\widehat{P}}{\overline{P}} \rmd P_n \stepb{\le} \frac{1}{2}\int \log\frac{\widehat{P}+\overline{P}}{2\overline{P}}\rmd P_n \\
&= \frac{1}{2}\int \log\frac{\widehat{P}+\overline{P}}{2\overline{P}}\rmd (P_n - \overline{P}) - \frac{1}{2}\KL\bpth{\overline{P} \Big\| \frac{\widehat{P}+\overline{P}}{2}} \\
&\stepc{\le} \frac{1}{2}\int \log\frac{\widehat{P}+\overline{P}}{2\overline{P}}\rmd (P_n - \overline{P}) - \frac{1}{2}H^2\bpth{\overline{P}, \frac{\widehat{P}+\overline{P}}{2}}, 
\end{align*}
where (a) follows from the definition of $\widehat{P}$ and $\overline{P}\in \calP$, (b) is due to the concavity of $x\mapsto \log x$, and (c) uses the inequality $\KL \ge H^2$. Therefore, 
\begin{align}\label{eq:basic_inequality}
H^2(\widehat{P}, \overline{P}) \stepd{\le} 16H^2\bpth{\overline{P}, \frac{\widehat{P}+\overline{P}}{2}} \le 16\int \log\frac{\widehat{P}+\overline{P}}{2\overline{P}}\rmd (P_n - \overline{P}) + \frac{32}{n}\log\frac{1}{\beta}, 
\end{align}
where (d) follows from \citep[Lemma 4.2]{geer2000empirical}. This is called the ``basic inequality'' in \citep{geer2000empirical}. 

The next step is to prove a high-probability upper bound on the integral in \eqref{eq:basic_inequality}, using empirical processes. For $P\in \calP$, let $Z_i(P) = \frac{1}{2}\log \frac{P+\overline{P}}{2\overline{P}}(X_i)$. For the subexponential norm $\rho^2(g):=2\bE[e^{|g|}-1-|g|]$, it holds that
\begin{align*}
\frac{1}{n}\sum_{i=1}^n \rho^2(Z_i(P)) &\stepe{\le} \frac{1}{n}\sum_{i=1}^n 8\bE\bqth{\bpth{ \bpth{\frac{P+\overline{P}}{2\overline{P}}(X_i)}^{1/2}-1 }^2} \\
&= 8\bE_{\overline{P}}\bqth{ \bpth{ \bpth{\frac{P+\overline{P}}{2\overline{P}}(X_i)}^{1/2}-1 }^2 } = 8H^2\pth{\frac{P+\overline{P}}{2}, \overline{P}} \le 4H^2(P,\overline{P}), 
\end{align*}
where (e) uses \citep[Lemma 7.1]{geer2000empirical} and that $Z_i(P)\ge -\frac{1}{2}\log 2$. In addition, for any bracket $[P^{\rL}, P^{\rU}]$, we have $Z_i(P^{\rL}) \le Z_i(P^{\rU})$, and the similar arguments to \citep[Lemma 7.3]{geer2000empirical} yield
\begin{align*}
\frac{1}{n}\sum_{i=1}^n \rho^2(Z_i(P^{\rU}) - Z_i(P^{\rL})) &\le \frac{1}{n}\sum_{i=1}^n \bE\bqth{\bpth{ \bpth{\frac{P^{\rU}+\overline{P}}{P^{\rL}+\overline{P}}(X_i)}^{1/2}-1 }^2} \\
&= \bE_{\overline{P}}\bqth{\bpth{ \bpth{\frac{P^{\rU}+\overline{P}}{P^{\rL}+\overline{P}}(X_i)}^{1/2}-1 }^2} \\
&\le 2H^2\pth{\frac{P^{\rU} + \overline{P}}{2}, \frac{P^{\rL} + \overline{P}}{2}} \le H^2(P^{\rU}, P^{\rL}). 
\end{align*}
Therefore, compared with the Hellinger bracketing number $\bracketing(\varepsilon,\calP(\delta),H)$, the conditions of \citep[Definition 8.1]{geer2000empirical} are fulfilled with $(\delta, R) = (\varepsilon, 2\delta)$ and $F^c = \varnothing$. Since
\begin{align*}
\frac{1}{n}\sum_{i=1}^n Z_i(P) - \bE\bqth{\frac{1}{n}\sum_{i=1}^n Z_i(P)} = \frac{1}{2}\int \log\frac{P+\overline{P}}{2\overline{P}}\rmd (P_n-\overline{P})
\end{align*}
coincides with the integral in the basic inequality \eqref{eq:basic_inequality}, now \citep[Theorem 8.13]{geer2000empirical} (with $R=2^{s+1}\delta$ and $a = c2^{2s}\delta^2\sqrt{n}$) yields
\begin{align*}
\bP\bpth{ \exists P\in \calP: \Big|\int \log\frac{P+\overline{P}}{2\overline{P}}\rmd (P_n-\overline{P})\Big| \ge c2^{2s}\delta^2 \wedge H(P, \overline{P})\le 2^{s+1}\delta}\le C\exp\bpth{-\frac{n(2^s\delta)^2}{C^2}}
\end{align*}
for every $s\ge 0$ with $2^s\delta \le 1$. Here the conditions in \citep[Eqn. (8.43) and (8.44)]{geer2000empirical} are clearly fulfilled for a small enough $c>0$, and the condition
\begin{align*}
c2^{2s}\delta^2\sqrt{n} \ge C_0\bpth{ \int_{c2^{2s-6}\delta^2}^{2^{s+1}\delta} \sqrt{\log \bracketing(u,\calP(\overline{P}, 2^{s+1}\delta),H)}\rmd u \vee 2^{s+1}\delta}
\end{align*}
required in \citep[Eqn. (8.45)]{geer2000empirical} is also satisfied by the definition of $\delta_n$, the assumption $\delta\ge \delta_n$, and adjusting constants.

Finally, for $\delta_0^2 := \delta^2 + \frac{64}{n}\log\frac{1}{\beta}$, by summing over $s=0,1,\dots, S$ where $S = \min\{s: 2^{s}\delta_0  > 1\}$, a standard peeling argument similar to \citep[Theorem 7.4]{geer2000empirical} gives 
\begin{align*}
\bP\bpth{H^2(\widehat{P},\overline{P}) \geq \delta_0^2}& \le \sum_{s=0}^S \bP\bpth{2^{s}\delta_0 \le H(\widehat{P},\overline{P}) \le 2^{s+1}\delta_0} \\
&\overset{\prettyref{eq:basic_inequality}}{\leq} \sum_{s=0}^S  \bP\bpth{\int \log\frac{\widehat{P}+\overline{P}}{2\overline{P}}\rmd (P_n - \overline{P})\geq \frac{1}{16}\bpth{2^{2s}\delta_0^2 - \frac{32}{n}\log \frac{1}{\beta}}  \wedge H(\widehat{P},\overline{P}) \le 2^{s+1}\delta_0 } \\
&\leq \sum_{s=0}^S \bP\bpth{\exists P \in \mathcal{P}: \Big|\int \log \frac{P + \overline{P}}{2\overline{P}}\rmd(P_n - \overline{P})\Big| \geq \frac{1}{32}2^{2s}\delta_0^2 \wedge H(P,\overline{P})\leq 2^{s+1}\delta_0} \\
&\leq \sum_{s=0}^\infty C\exp(-\frac{n(2^s\delta_0)^2}{C^2}) \leq C_0\exp(-\frac{n\delta_0^2}{C_0^2}),
\end{align*}
for a suitably chosen universal constant $C_0 > 0$. 
\subsection{Proof of \Cref{thm:metric-entropy}}\label{append:metric-entropy}
In this section, we prove \Cref{thm:metric-entropy} based on \Cref{lemma:key-covering}. The proof decomposes into three steps. 

\subsubsection{$L_\infty$ covering for scale mixtures}
Our first step is to establish an upper bound on the covering number of \emph{pure} scale mixtures
\begin{align*}
\calP_0 = \{f_{0,G}: \supp(G)\subseteq [\sigmin, \sigmax]\}
\end{align*}
under the $L_\infty$ norm. 

\begin{lem}\label{lemma:step-I}
For $\varepsilon\in (0,1/2)$ and $\sigmin \le 1$, 
\begin{align*}
    \log N(\varepsilon,\calP_0,L_\infty) \le C\log^6\pth{\frac{1}{\varepsilon\sigmin}}\log\log \pth{\frac{1}{\varepsilon\sigmin}}. 
\end{align*}
\end{lem}
\begin{proof}
The proof relies on \Cref{lemma:key-covering}, with a suitable reparametrization and truncation. Let $t=1/\sigma$ and $H$ be the pushforward measure of $G$ under $\sigma\mapsto 1/\sigma$, then $H$ is supported on $[0, \sigmin^{-1}]$, and 
\begin{align*}
f_{0,G}(x) = \bE_{\sigma\sim G}\bqth{\frac{1}{\sqrt{2\pi}\sigma}\exp(-\frac{x^2}{2\sigma^2})} = \bE_{t\sim H}\bqth{\frac{t}{\sqrt{2\pi}}\exp(-\frac{t^2x^2}{2})} =: \widetilde{f}_H(x). 
\end{align*}
To construct an $\varepsilon$-cover of $\{\widetilde{f}_H: \supp(H)\subseteq [0,\sigmin^{-1}]\}$, we invoke \Cref{lemma:key-covering} with parameters
\begin{align} \label{eq:parameter_choices}
t_{\min} = \frac{\varepsilon}{4}, \quad t_{\max} = \frac{1}{\sigmin}, \quad x_{\min} = c_0(\varepsilon \sigmin^3)^{1/2}, \quad 
x_{\max} = C_0\varepsilon^{-1}\sqrt{\log(1/\varepsilon)}, 
\end{align}
where $c_0,C_0>0$ are appropriate universal constants. By \Cref{lemma:key-covering}, there exists a cover $\calH$ with $|\calH| = \exp(O(\log^3\frac{1}{\varepsilon\sigmin}\log^4\log\frac{1}{\varepsilon\sigmin}))$ such that for every $H$ supported on $[t_{\min},t_{\max}]$, there exists $H_0\in \calH$ such that
\begin{align*}
\sup_{x\in [x_{\min},x_{\max}]} \Big| \widetilde{f}_H(x) - \widetilde{f}_{H_0}(x) \Big| \le \frac{\varepsilon}{4}. 
\end{align*}
% We show that $\{\widetilde{f}_{H_0}(x): H_0\in \calH\}$ remains an $\varepsilon$-cover even for $\supp(H)\subseteq [0,t_{\max}]$ and $x\notin [x_{\min}, x_{\max}]$. 
We show that a similar inequality still holds even if $x\notin [x_{\min},x_{\max}]$: 
\begin{enumerate}
\item If $x\ge x_{\max}$, the function
$t\mapsto t e^{-t^2 x^2/2}$ is decreasing for $x\ge 1/t_{\min}$, hence 
\begin{align*}
0\le \widetilde{f}_H(x)
\le \frac{t_{\min}}{\sqrt{2\pi}}\exp\bpth{-\frac{t_{\min}^2 x^2}{2}}
\le \frac{\varepsilon}{4}
\end{align*}
for a large constant $C_0$. The same bound also holds for $\widetilde{f}_{H_0}$, so that $| \widetilde{f}_H(x) - \widetilde{f}_{H_0}(x)| \le \frac{\varepsilon}{4}$. 
\item If $0\le x\le x_{\min}$, note that
\begin{align*}
|\widetilde{f}_H'(x)| = x\bE_{t\sim H}\bqth{\frac{t^3}{\sqrt{2\pi}}\exp(-\frac{t^2x^2}{2})} \le \frac{x_{\min}}{\sqrt{2\pi}\sigmin^3}.
\end{align*}
Therefore, for a small constant $c_0$, 
\begin{align*}
&|\widetilde{f}_H(x) - \widetilde{f}_{H_0}(x)| \\
&\le |\widetilde{f}_H(x) - \widetilde{f}_H(x_{\min})| + |\widetilde{f}_H(x_{\min}) - \widetilde{f}_{H_0}(x_{\min})| + |\widetilde{f}_{H_0}(x_{\min}) - \widetilde{f}_{H_0}(x)| \\
&\le 2\cdot \frac{x_{\min}^2}{\sqrt{2\pi}\sigmin^3} + \frac{\varepsilon}{4} \le \frac{\varepsilon}{2}. 
\end{align*}
\item If $x\le 0$, since $\widetilde{f}_H(x)$ is an even function, the statement follows from symmetry. 
\end{enumerate}
Therefore, we have established that if $\supp(H)\subseteq [t_{\min}, t_{\max}]$, then $\|\widetilde{f}_H - \widetilde{f}_{H_0}\|_\infty \le \frac{\varepsilon}{2}$. 

Next we extend to the case where $\supp(H)$ could be $[0,t_{\max}]$. Note that if $\supp(H)\subseteq [0,t_{\min}]$, then $\widetilde{f}_H(x)\le t_{\min}= \frac{\varepsilon}{4}$. Therefore, we can modify the cover $\calH$ by a new cover
\begin{align}
\calH' = \bsth{ (1-w)\delta_0 + wH_0: H_0\in \calH, w\in\{0, \delta, 2\delta, \dots, 1 \} }
\end{align}
with $\delta = \frac{\varepsilon}{4t_{\max}}$, so that $|\calH'| = O(|\calH|/\delta)$. Now for any $H$ supported on $[0,t_{\max}]$, we write $H=(1-w)H_1 + wH_2$ with $\supp(H_1)\subseteq [0,t_{\min}), \supp(H_2)\subseteq [t_{\min},t_{\max}]$, and $w\ge 0$. Then for $H' = (1-w')\delta_0 + w'H_0 \in \calH'$ with $|w-w'|\le \delta$ and $H_2$ approximated by $H_0\in \calH$, we get
\begin{align*}
\|\widetilde{f}_H - \widetilde{f}_{H'}\|_\infty &= \|(1-w)\widetilde{f}_{H_1} + w\widetilde{f}_{H_2} - w'\widetilde{f}_{H_0}\|_\infty \\
&\le \|\widetilde{f}_{H_1}\|_\infty + \|\widetilde{f}_{H_2}-\widetilde{f}_{H_0}\|+ |w-w'|\cdot  \|\widetilde{f}_{H_0}\|_\infty \\
&\le \frac{\varepsilon}{4} + \frac{\varepsilon}{2} + \delta \frac{t_{\max}}{\sqrt{2\pi}} \le \varepsilon. 
\end{align*}
This shows that $\{\widetilde{f}_{H'}: H' \in \calH'\}$ is an $\varepsilon$-cover of $\{\widetilde{f}_H: \supp(H)\subseteq [0,\sigmin^{-1}]\}$ under $L_\infty$, with
\begin{align*}
\log |\calH'| = \log |\calH| + O\bpth{\log \frac{1}{\delta}} = O\bpth{\log^3\frac{1}{\varepsilon\sigmin}\log^4\log\frac{1}{\varepsilon\sigmin}}. 
\end{align*}
\end{proof}

\subsubsection{Hellinger bracketing for scale mixtures} \label{app:proof_bracketing_vs_covering}

For the class of scale mixtures $\calP_0$, by the standard approach in \citep{ghosal2001entropies}, we can transform an $L_\infty$-cover into a Hellinger bracket. 
\begin{lem}
\label{lemma:step-II}
For $\varepsilon\in (0,1/2)$, 
\begin{align*}
\log \bracketing(\varepsilon, \calP_0, H) \le C\log^6\bpth{\frac{\sigmax}{\varepsilon \sigmin}}\log\log\bpth{\frac{\sigmax}{\varepsilon \sigmin}}. 
\end{align*}
\end{lem}

\begin{proof}
Since the bracketing number is invariant by scaling $(\sigmin, \sigmax)$ simultaneously, WLOG we may assume that $\sigmin\le 1$ so that \Cref{lemma:step-I} can be applied. First, note that for (not necessarily probability) measures $\mu$ and $\nu$, 
\begin{align*}
H^2(\mu, \nu) = \int (\sqrt{\rmd \mu} - \sqrt{\rmd \nu})^2 \le \int |\rmd \mu - \rmd \nu| = \|\mu - \nu\|_1. 
\end{align*}
Therefore, $\log \bracketing (\varepsilon, \calP_0, H) \le \log \bracketing(\varepsilon^2, \calP_0, L_1)$. To construct a $L_1$-bracket, let $p_1, \dots, p_N$ be an $\eta$-cover of $\mathcal{P}_0$ in $L_{\infty}$, with $\eta>0$ to be specified later. 
% By \Cref{lemma:step-I}, we can take $\log N = O(\log^c(\frac{1}{\eta\sigmin})) = O(\log^c(\frac{\sigmax}{\varepsilon\sigmin}))$. 
We then construct the brackets $\{[P_i^L, P_i^U]: i \in [N]\}$ by $P_i^\rL := \max\{p_i - \eta, 0\}, P_i^\rU := \min\{p_i + \eta, R\}$,
with $R(x) = \frac{1}{\sigmin}\exp(-\frac{x^2}{2\sigmax^2})$. Since every $p\in \calP_0$ satisfies
\begin{align*}
p(x) = \bE_{\sigma\sim G}\bqth{\frac{1}{\sigma}\exp(-\frac{x^2}{2\sigma^2})} \le \frac{1}{\sigmin}\exp(-\frac{x^2}{2\sigmax^2}) = R(x), 
\end{align*}
the inequality $\|p-p_i\|_\infty \le \eta$ indeed implies $P_i^\rL \le p \le P_i^\rU$. To upper bound the $L_1$ norm, for every $B \geq 0$ we have
\begin{align*}
    \int_{\R}|P_{i}^\rL(x) - P_{i}^\rU(x)|\rmd x \leq \int_{|x|\geq B}R(x)\rmd x + 2\eta B \leq \frac{2\sigma_{\max}^2}{B\sigma_{\min}}\exp(-\frac{B^2}{2\sigmax^2}) + 2\eta B,
\end{align*}
where the last inequality uses the Mills' ratio upper bound. Therefore, by choosing 
\begin{align*}
B = C_0\sigmax\sqrt{\log(\frac{\sigmax}{\sigmin \varepsilon^2})}, \qquad \eta = \frac{\varepsilon^2}{C_0B}
\end{align*}
with a large universal constant $C_0>0$, we ensure that $\|P_i^\rU - P_i^\rL\|_1 \le \varepsilon^2$. Finally, by \Cref{lemma:step-I}, we can take $\log N = O(\log^3(\frac{1}{\sigmin\eta})\log^4\log (\frac{1}{\sigmin\eta}))=O(\log^3(\frac{\sigmax}{\sigmin \varepsilon})\log^4\log(\frac{\sigmax}{\sigmin \varepsilon}))$. 
\end{proof}

\subsubsection{Local Hellinger bracketing for location-scale mixtures}

Finally, we extend the Hellinger bracketing for the smaller family $\calP_0$ in \Cref{lemma:step-II} to a \emph{local} Hellinger bracketing for the entire family $\calP = \{f_{\mu,G}: \mu\in \mathbb{R}, \supp(G)\subseteq [\sigmin, \sigmax]\}$ involving the location parameter in \Cref{thm:metric-entropy}. \\

\begin{proof}[Proof of \Cref{thm:metric-entropy}]
Let $\overline{P} = f_{\mu_0,G_0}\in \calP$ be the center of the local ball; by translation invariance we may assume that $\mu_0 = 0$. We first establish an upper bound on $|\mu|$ for every $f_{\mu,G}\in \calP(\overline{P},\delta)$. In fact, by the symmetrization inequality in \Cref{lemma:symmetrization}, we have
\begin{align*}
\delta \ge H(f_{0,G_0}, f_{\mu,G}) \ge \frac{1}{2}H(f_{\mu,G_0}, f_{-\mu,G_0}). 
\end{align*}
By the variational lower bound of the Hellinger distance in \eqref{eq:Hellinger-variational}, we get
\begin{align*}
H^2(f_{\mu,G_0}, f_{-\mu,G_0}) \ge \frac{1}{4}(f_{\mu,G_0}([0,\infty)) - f_{-\mu,G_0}([0,\infty)))^2 \ge \frac{1}{4}\bpth{2\Phi(\frac{|\mu|}{\sigmax})-1}^2, 
\end{align*}
where $\Phi$ is the CDF of $\calN(0,1)$. For $\delta \le 1/8$, solving $|\mu|$ from the above two inequalities yields $|\mu| \le C\sigmax$ for a universal constant $C>0$, and therefore
\begin{align*}
\calP(\overline{P}, \delta) \subseteq \sth{ f_{\mu,G}: |\mu|\le C \sigmax, \supp(G)\subseteq [\sigmin, \sigmax] } =: \calP_{\mathrm{loc}}.  
\end{align*}
To construct a Hellinger bracket for $\calP_{\mathrm{loc}}$, we start from the Hellinger bracket $\{[P_i^\rL, P_i^\rU]: i\in [N]\}$ for $\calP_0$ constructed in \Cref{lemma:step-II}, with $\|P_i^\rU - P_i^\rL\|_1\le \frac{\varepsilon^2}{4}$. Let $\eta>0$ be a parameter to be specified later, and $\{\mu_1,\dots,\mu_M\}$ be a uniform discretization of $[-C \sigmax, C \sigmax]$ with spacing $2\eta$. Here $M=O(\frac{\sigmax}{\eta})$, and for $i\in [N], j\in [M]$, define
\begin{align*}
P_{i,j}^{\rL}(x) = \inf_{|z|\le \eta} P_i^\rL(x-\mu_j-z), \quad P_{i,j}^{\rU}(x) = \sup_{|z|\le \eta} P_i^\rU(x-\mu_j-z). 
\end{align*}
We claim that $\{[P_{i,j}^\rL, P_{i,j}^\rU]: i\in [N], j\in [M]\}$ is a bracket for $\calP_{\mathrm{loc}}$. In fact, for any $f_{\mu,G}\in \calP_{\mathrm{loc}}$, assume $P_i^\rL \le f_{0,G} \le P_i^\rU$ and $|\mu-\mu_j|\le \eta$. For this pair $(i,j)$, we have
\begin{align*}
f_{\mu,G}(x) = f_{0,G}(x-\mu) \ge \inf_{|z|\le \eta} f_{0,G}(x-\mu_j-z) \ge \inf_{|z|\le \eta} P_i^\rL(x-\mu_j-z) = P_{i,j}^\rL(x), 
\end{align*}
and similarly $f_{\mu,G} \le P_{i,j}^\rU$. To upper bound the $L_1$ norm of this bracket, recall that $P_i^\rL = \max\{p_i-\eta', 0\}$ for some $p_i\in \calP_0$ and $\eta'>0$ in the proof of \Cref{lemma:step-II}. Therefore, 
\begin{align*}
\sup_{|z|\le \eta} P_i^\rL(x-\mu_j-z) - \inf_{|z|\le \eta} P_i^\rL(x-\mu_j-z) &\le 2\eta\cdot \sup_{|z|\le \eta} |p'(x-\mu_j-z)| \\
&\le \frac{2\eta}{\sqrt{2\pi}\sigmin^2}\exp(-\frac{(|x-\mu_j|-\eta)_+^2}{2\sigmax^2}), 
\end{align*}
where the last step follows from simple algebra, with $x_+ = \max\{x,0\}$. Consequently, 
\begin{align}
\Big\| \sup_{|z|\le \eta} P_i^\rL(\cdot -\mu_j-z) - \inf_{|z|\le \eta} P_i^\rL(\cdot -\mu_j-z) \Big\|_1 &\le \frac{2\eta}{\sqrt{2\pi}\sigmin^2} \int_{\bR} \exp(-\frac{(|x-\mu_j|-\eta)_+^2}{2\sigmax^2}) \rmd x \nonumber \\
&\le \frac{C\eta}{\sigmin^2}\bpth{\eta + \sigmax}. \label{eq:ineq-1}
\end{align}

On the other hand, since both functions $P_i^\rU$ and $P_i^\rL$ constructed in the proof of \Cref{lemma:step-II} are even and non-increasing on $[0,\infty)$, we have
\begin{align*}
\sup_{|z|\le \eta} P_i^\rU(x -\mu_j-z)  = P_i^\rU((|x-\mu_j|-\eta)_+),
\end{align*}
and similarly for $P_i^\rL$. Therefore, integrating separately over two regimes $|x-\mu_j|\le \eta$ and $|x-\mu_j|>\eta$ yields
\begin{align}
    & \Big\| \sup_{|z|\le \eta} P_i^\rU(\cdot -\mu_j-z) - \sup_{|z|\le \eta} P_i^\rL(\cdot -\mu_j-z)\Big\|_1 \nonumber \\
    &= 2\eta |P_i^\rU(0) - P_i^\rL(0)| + \|P_i^\rU - P_i^\rL\|_1 
    \le 4\eta \eta' + \frac{\varepsilon^2}{4} \le \frac{\varepsilon^2}{2}, \label{eq:ineq-2}  
\end{align}
where the first inequality follows from $\|P_i^\rU - P_i^\rL\|_\infty \le 2\eta'$ and $\|P_i^\rU - P_i^\rL\|_1 \le \frac{\varepsilon^2}{4}$, and the second inequality follows from the choice of $\eta' \le \frac{\varepsilon^2}{4}$ in the proof of \Cref{lemma:step-II} and $\eta\le \frac{1}{4}$. Now by the triangle inequality and \eqref{eq:ineq-1}, \eqref{eq:ineq-2}, we obtain $
\|P_{i,j}^\rU - P_{i,j}^\rL\|_1 \le \varepsilon^2$ as long as
\begin{align*}
\eta = c_0\bpth{\varepsilon\sigmin \wedge \frac{\varepsilon^2\sigmin^2}{\sigmax}}
\end{align*}
for a small universal constant $c_0>0$. Since $H^2(\mu,\nu)\le \|\mu-\nu\|_1$, this implies that $\{[P_{i,j}^\rL, P_{i,j}^\rU]: i\in [N], j\in [M]\}$ is an $\varepsilon$-Hellinger bracket with size at most
\begin{align*}
\exp\bpth{O\bpth{\log^6(\frac{\sigmax}{\varepsilon\sigmin})\log\log(\frac{\sigmax}{\varepsilon\sigmin}) + \log \frac{\sigmax}{\eta}}} = \exp\bpth{O\bpth{\log^6(\frac{\sigmax}{\varepsilon\sigmin})\log\log(\frac{\sigmax}{\varepsilon\sigmin}) }}. 
\end{align*}
This completes the proof. 
\end{proof}

\subsection{Proof of \cref{lemma:key-covering}} \label{app:proof_covering_scale_mixture}

%Recall that we need to bound the $L_{\infty}$ covering number of the function class
%\begin{align*}
%\calF = \bsth{ f_H(x) = \bE_{t\sim H}\bqth{\frac{t}{\sqrt{2\pi}}\exp(-\frac{t^2 x^2}{2})}: \supp(H)\subseteq [t_{\min}, t_{\max}] } 
%\end{align*}
%on the domain $x\in [x_{\min}, x_{\max}]$.
From \cref{lemma:generalized-moment-matching}, it suffices to find functions $a_1(t), \dots, a_L(t)$ and $g_1(x), \dots, g_L(x)$ such that $|a_k(t)| + t\abs{a'_k(t)} \le A$ for all $t\in [t_{\min}, t_{\max}]$, $|g_k(x)|\le G$ for all $x\in [x_{\min}, x_{\max}]$, and
\begin{align} \label{eq:approximation_goal_1}
\sup_{x\in [x_{\min},x_{\max}], t\in [t_{\min},t_{\max}]} \Big| t e^{-\frac{t^2x^2}{2}} - \sum_{k=1}^L a_k(t) g_k(x) \Big| \le \varepsilon. 
\end{align}

Recall from \cref{lem:chebyshev_approximation}, for the function $h(v) = e^{-K e^{\lambda v}}$ on $v\in[-1,1]$, we have the approximation 
\begin{align}
    \abs{h(v) - P_L(v) } \leq \varepsilon' 
\end{align}
for all $v \in [-1,1]$ where $L = O(\log(1/\varepsilon') + \lambda \log(\lambda/\varepsilon'))$ and $P_L$ is a polynomial of degree $L$ in $v$. 
Recall that the degree-$L$ Chebyshev approximation of $h$ corresponds to the polynomial $P_L(v) = \sum_{j=0}^L c_j T_j(v)$,
where $T_j(x)$ is the degree-$j$ Chebyshev polynomial with $T_j(\cos( \theta)) = \cos(j \theta)$, and $c_j$ is the Chebyshev coefficient of $h$ defined as 
\begin{align}\label{eq:chebyshev-coeff}
    c_0 = \frac{1}{\pi}\int_0^\pi h(\cos(\theta))\rmd \theta, \quad c_j = \frac{2}{\pi} \int_{0}^{\pi} h(\cos(\theta)) \cos(j\theta) \rmd \theta, \quad j\ge 1. 
\end{align}

Since $T_j$ are polynomials in $v$, each of degree at most $L$, we can write $ P_L(v) = \sum_{j=0}^{L} p_j v^j $. 
% since each of the terms in the Chebyshev expansion is a polynomial in $u$ of degree at most $L$. 
Now set $u := \log t + \log x$, so that $u\in[u_{\min},u_{\max}]$ where $u_{\min} = \log(t_{\min}x_{\min})$ and $u_{\max} = \log(t_{\max} x_{\max})$. Define the affine rescaling
\begin{align*}
    v(u) := \frac{2u - (u_{\min}+u_{\max})}{u_{\max}-u_{\min}} \in [-1,1].
\end{align*}
Then
\begin{align*}
    e^{-\frac{t^2x^2}{2}}
    = \exp\!\left(-\frac{1}{2}e^{2u}\right)
    = \exp\!\left(-K e^{\lambda v(u)}\right),
\end{align*}
where $K = \frac{1}{2}e^{u_{\min}+u_{\max}} = \frac{1}{2}t_{\min}t_{\max}x_{\min}x_{\max}$ and
$\lambda = u_{\max}-u_{\min} = \log\frac{t_{\max}x_{\max}}{t_{\min}x_{\min}}$. In particular,
\begin{align*}
L = O\bpth{ \log\frac{1}{\varepsilon'} + \lambda \log \frac{\lambda}{\varepsilon'} } = O\bpth{ \log^2 \frac{t_{\max}x_{\max}}{t_{\min}x_{\min}\varepsilon} }. 
\end{align*}

Therefore, applying the Chebyshev approximation of $h(v)=e^{-K e^{\lambda v}}$ and substituting $v=v(u)$ yields
\begin{align}
    \abs{ e^{-\frac{t^2 x^2 }{2}} - \sum_{i=0}^{L} p_i \, v(u)^{i} } \leq \varepsilon'
\end{align}
for all $t \in [t_{\min}, t_{\max}]$ and $x \in [x_{\min}, x_{\max}]$.
Since $v(u)$ is affine in $u$, we can expand $\sum_{i=0}^{L} p_i v(u)^{i}$ into a degree-$L$ polynomial in $u$, i.e.,
\begin{align*}
    \sum_{i=0}^{L} p_i v(u)^i = \sum_{i=0}^L \widetilde{p}_i u^i,
\end{align*}
and then expanding the term $\left( \log t + \log x \right)^{i}$ using the binomial theorem, we have
% \begin{align}
%     \abs{ e^{-\frac{t^2 x^2 }{2}} - \sum_{i=0}^{L} p_i \sum_{j=0}^{i} \binom{i}{j} (\log t)^{j} (\log x)^{i-j} } \leq \epsilon.
% \end{align}
% which can be rewritten as
\begin{align}
    \abs{ t e^{-\frac{t^2 x^2 }{2}} - \sum_{i=0}^{L} \sum_{j=0}^{i} \widetilde{p}_i \binom{i}{j} (\log t)^{j} (\log x)^{i-j} t } \leq t_{\max} \varepsilon'.
\end{align}
We are left to bound the size of the coefficients $ \widetilde{p}_i \binom{i}{j} $. 
Since $v(u)$ is an affine change of variables, with coefficients $2/\lambda$ and $ \log(2K) / \lambda $, we have using the binomial theorem $ \abs{\widetilde{p_i}  } \leq \max_i \abs{p_i} L \cdot 2^{2L} ((\log(2 K) /\lambda )^L \land 1)   $. 
To bound $p_i$, first note that since $h(u) = e^{-K e^{\lambda u}}$ is a bounded function on $[-1,1]$, by \eqref{eq:chebyshev-coeff}, all Chebyshev coefficients have magnitude at most $2$. 
% $|c_k| \leq 2 \max_{u\in[-1,1]} |h(u)| \leq 2$.
Further, we have that the coefficients of the Chebyshev polynomial $T_k(u)$ are bounded by $2^{k-1}$ in absolute value (see e.g., \citep{trefethen2019approximation}). 
We also trivially have $\binom{i}{j} \leq 2^i \leq 2^L$. 
Finally, for the functions in the approximation, we have  $   \abs{t (\log t)^j }     \leq \max\{ \log^L(1/t_{\min}) ,  t_{\max}  \log^L( t_{\max}    ) \}  $ and $ \abs{\log x}^{i-j} \leq \max\{   \log^L( x_{\max}   )   , \log^L(1/x_{\min}) \} $ for all $t\in [t_{\min}, t_{\max}]$ and $x\in [x_{\min}, x_{\max}]$.
% Define $ \hat{t}_{\max} = t_{\max} $ if $t_{\max} >1 $ and ${t}^{-1}_{\max}$ if $ t_{\max} \leq 1 $. 
% Similarly, define $\hat{x}_{\min}$. 
% Note that we can assume that $x_{\max} >1 $ and $t_{\min} \leq 1$ as can be seen from \cref{eq:parameter_choices}. 
Combining these three bounds, we have that an approximation of the form in \eqref{eq:approximation_goal_1} holds with bound on the size of the terms given by $A\cdot G \leq L 2^{4L} {t}_{\max}   (\log( {t}_{\max} / t_{\min}) )^{2L}  (\log( x_{\max} / {x}_{\min}) )^{2L} $. 
Further, there are at most $L^2$ such terms in the approximation.  
Setting $\varepsilon' = \varepsilon / t_{\max}$, and plugging into \cref{lemma:generalized-moment-matching}, we conclude that the $L_{\infty}$ covering number of $\mathcal{F}$ is bounded by
\begin{align*}
& \log N(\varepsilon, \mathcal{F}, L_{\infty}([x_{\min}, x_{\max}])) \\ &\leq O\bpth{ L^2 \log\frac{A G L^2 \log(t_{\max} / t_{\min})}{\varepsilon'} } \\
&= O\bpth{ L^2 \log\bpth{ \frac{ 2^{4L} t^2_{\max}   (\log( {t}_{\max} / t_{\min}) )^{2L+1}  (\log( x_{\max} / {x}_{\min}) )^{2L} L^3  } { \varepsilon} }  } \\ 
% & = O( L^2 + L \log(t_{\max}/\epsilon) ) \\ 
 &= O \bpth{ \log^6 \bpth{\frac{x_{\max } t_{\max} }{ x_{\min} t_{\min} \varepsilon  } } \cdot \log \log \bpth{ \frac{ x_{\max } {t}_{\max} }{ {x}_{\min} t_{\min} \varepsilon  } }    } 
\end{align*}
as required. 

% Multiplying both sides by $t

% From \cref{lem:chebyshev_approximation}, we have that 

% We combine \cref{lemma:generalized-moment-matching} and \cref{lem:chebyshev_approximation} to finish the proof of \cref{lemma:key-covering}.

% \subsection{Proof of \cref{thm:metric-entropy}}

% \subsection{Finishing the Proof of \cref{lemma:density_estimation}} \label{sec:final_proof}

\subsection{Proof of \Cref{lemma:generalized-moment-matching}}
By Carath\'eodory's theorem, for any $H\in \calP([t_{\min}, t_{\max}])$, there exists a distribution $H'\in \calP([t_{\min}, t_{\max}])$ with at most $O(L)$ atoms such that $\bE_{t\sim H}[a_k(t)] = \bE_{t\sim H'}[a_k(t)]$ for all $k\in [L]$. By discretizing the support of $H'$ into a geometric grid $\{t_{\min}, e^{\delta} t_{\min}, \dots, t_{\max}\}$, and the weights of $H'$ into a $\delta$-net on the simplex under $\|\cdot \|_1$, with $\delta = \frac{\varepsilon}{2AGL}$, this forms a finite set $\calH$ of size $\exp(O(L\log\frac{AGL\log(t_{\max}/t_{\min})}{\varepsilon}))$. In addition, if $H'=\sum_{i=1}^m w_i\delta_{t_i}$ and $H''=\sum_{i=1}^m w_i'\delta_{t_i'}$ with $\max_{i\in [m]}|\log t_i-\log t_i'| \le \delta$ and $\|w-w'\|_1 \le \delta$, we have
\begin{align*}
 |\bE_{t\sim H'}[a_k(t)] - \bE_{t\sim H''}[a_k(t)]| 
& \le \Big| \sum_{i=1}^m w_i(a_k(t_i) - a_k(t_i')) \Big| + \Big| \sum_{i=1}^m (w_i-w_i')a_k(t_i') \Big| \\
&\le \sum_{i=1}^m w_i A|\log t_i - \log t_i'| + \|w-w'\|_1\cdot A \le 2\delta A. 
\end{align*}
Since $\delta = \frac{\varepsilon}{2AGL}$ and $|g_k(x)|\le G$, we conclude that any $H'\in \calP([t_{\min},t_{\max}])$ with at most $O(L)$ atoms can be approximated by some $H''\in \calH$ such that
\begin{align*}
\sup_{x\in [x_{\min},x_{\max}], t\in [t_{\min},t_{\max}]} \Big| \sum_{k=1}^L \bE_{t\sim H'}[a_k(t)] g_k(x) - \sum_{k=1}^L  \bE_{t\sim H''}[a_k(t)] g_k(x) \Big| \le \varepsilon. 
\end{align*}
By triangle inequality, this finite set $\calH$ induces a $2\varepsilon$-covering of $\calF$.

\section{Bounding Modulus of Continuity}\label{sec:func_ineq_overall}
%%%%%%%%%%%%%%%%%%%%%%%%%%%%%%%%%%%%%%%%%%%%%%%%%%%%%%%%%%%%%%%%%

In this section, we establish upper bounds of the Hellinger modulus of continuity in \eqref{eq:modulus-of-continuity}, with various choices of $G_n$. 
Specifically, we prove the upper bound in \Cref{lemma:functional_inequality} for the subset-of-signals problem, and upper bounds in \Cref{cor:examples} for explicit examples of $G_n$.

\subsection{Tightness of Lemma \ref{lemma:functional_inequality}}
Fix any $t\in (0,1)$ and $p\in (0,\frac{1}{2})$. To verify the tightness of \Cref{lemma:functional_inequality}, we show that there exist universal constants $c_1, c_2>0$ such that if
\begin{align*}
G = p\delta_1 + (1-p)\delta_{\sigma}, \quad \mu = c_1 \begin{cases}
    (\frac{t^3}{p^4})^{1/6} & \text{if } t \le p^{4/3} \\
    (\frac{t^3}{p^4})^{1/2} & \text{if } p^{4/3}\le t\le p
\end{cases}, \quad \sigma = \frac{c_2\mu}{\sqrt{t}},
\end{align*}
then $H^2(f_{\mu,G},f_{-\mu,G})\le t$. Indeed, by simple algebra, we have
\begin{align*}
H^2(f_{\mu,G}, f_{-\mu, G}) &\le \int_{-\infty}^\infty \frac{(f_{\mu,G} - f_{-\mu, G})^2}{f_{\mu,G}+f_{-\mu, G}} \rmd x \\
&\le \frac{1}{1-p}\int_{-\infty}^\infty \frac{(f_{\mu,G}(x) - f_{-\mu, G}(x))^2}{\varphi_\sigma(x-\mu) + \varphi_\sigma(x+\mu)} \rmd x \\
&\le \frac{e^{\mu^2/\sigma^2}}{2(1-p)}\int_{-\infty}^\infty \frac{(f_{\mu,G}(x) - f_{-\mu, G}(x))^2}{\varphi_\sigma(x)} \rmd x \\
&\le \frac{e^{\mu^2/\sigma^2} p^2}{1-p} \int_{-\infty}^\infty \frac{(\varphi(x-\mu) - \varphi(x+\mu))^2}{\varphi_\sigma(x)}\rmd x \\
&\quad +  e^{\mu^2/\sigma^2} (1-p) \int_{-\infty}^\infty \frac{(\varphi_\sigma(x-\mu) - \varphi_\sigma(x+\mu))^2}{\varphi_\sigma(x)}\rmd x. 
\end{align*}
Here we use $\varphi$ and $\varphi_\sigma$ to denote the densities of $\calN(0,1)$ and $\calN(0,\sigma^2)$, respectively. To proceed, since $\sigma = c_2\mu/\sqrt{t}\ge c_2\mu$, we have $\mu^2/\sigma^2 \le 1$ whenever $c_2\ge 1$. For the remaining integrals, as long as $\tau^2 < 2\sigma^2$, simple algebra gives
\begin{align*}
\int_{-\infty}^\infty \frac{(\varphi_\tau(x-\mu) - \varphi_\tau(x+\mu))^2}{\varphi_\sigma(x)}\rmd x = \frac{2\sigma^2}{\tau\sqrt{2\sigma^2 - \tau^2}}\bqth{ \exp\bpth{\frac{\mu^2}{2\sigma^2-\tau^2}} - \exp\bpth{-\frac{\mu^2}{\tau^2}}}. 
\end{align*}
Plugging in $\tau\in \{1,\sigma\}$, and using that $\sigma \ge 1$ in both regimes, we get
\begin{align*}
H^2(f_{\mu,G}, f_{-\mu, G}) = O\bpth{ p^2 \sigma (\mu^2 \wedge 1) + \frac{\mu^2}{\sigma^2} } = O\bpth{ \frac{c_2p^2}{\sqrt{t}}(\mu^3 \wedge \mu) + \frac{t}{c_2^2} }. 
\end{align*}
Choosing $c_2>0$ large enough and $c_1>0$ small enough, we have $H^2(f_{\mu,G}, f_{-\mu, G}) \le t$, as desired. 

Finally, if $t\ge 3p$, by convexity of squared Hellinger distance, for every $\mu>0$ it holds that
\begin{align*}
H^2(f_{-\mu,G}, f_{\mu,G}) \le 2p + H^2(\calN(-\mu,\frac{c_2^2\mu^2}{t}),\calN(\mu,\frac{c_2^2\mu^2}{t})) \le 2p + \frac{t}{3}\le t, 
\end{align*}
 as long as $c_2>0$ is large enough. Therefore, $\omega_{H^2,G}(t) \ge 2\mu$ in this case, and letting $\mu\to\infty$ shows that $\omega_{H^2,G}(t)$ does not admit a uniform upper bound.

\subsection{Modulus of Continuity for SoS Priors: Proof of Lemma \ref{lemma:functional_inequality}}
%\label{sec:func_ineq}

The proof of \Cref{lemma:functional_inequality} relies on the following variational lower bound of squared Hellinger distance: 
\begin{align}\label{eq:Hellinger-variational}
    H^2(P,Q) = \int \frac{(P-Q)^2}{(\sqrt{P} + \sqrt{Q})^2} \geq \frac{1}{2}\int \frac{(P-Q)^2}{P+Q} = \frac{1}{2}\sup_{T}\frac{(\bE_P[T] - \bE_Q[T])^2}{\bE_P[T^2] + \bE_Q[T^2]},
\end{align}
where the supremum is over all test functions $T: \calX \to \bR$, and the last step is Cauchy--Schwarz. In the definition of Hellinger modulus of continuity, suppose the choices $P = f_{\mu, G}$ and $Q = f_{0, G}$ satisfy $H^2(P,Q)\le t$. Now choosing test function $T_{\Delta}(x) = 1_{[\mu, \mu + \Delta]}(x)$, the variational form gives
    \begin{align}\label{eq:lb}
\frac{\bqth{\bE_{G}\bpth{\Phi\pth{\frac{\Delta}{\sigma}} - \Phi(0)  - \Phi(\frac{\Delta + \mu}{\sigma}) +  \Phi(\frac{\mu}{\sigma})}}^2}{\bE_{G}\bpth{\Phi(\frac{\Delta}{\sigma}) - \Phi(0)+\Phi(\frac{\Delta + \mu}{\sigma}) - \Phi(\frac{\mu}{\sigma})}} \le 2t. 
\end{align}
Here $ \Phi(x) = \int_{-\infty}^x \varphi(t) dt$ denote the standard Gaussian CDF.
Taking $\Delta=\infty$, and noting that the denominator always lies in $[0,1]$, we obtain
\begin{align*}
\sqrt{2t} \ge \bE_G\bpth{\Phi(\frac{\mu}{\sigma}) - \Phi(0)} \ge \Phi'(1) \cdot \bE_G\bqth{\frac{ \mu\indc{\sigma\ge \mu} }{\sigma}}
\end{align*}
due to $\Phi(x)-\Phi(0)\ge \Phi'(1)x$ for $x\in [0,1]$. This shows that for a universal $c_1>0$, 
\begin{align}\label{eq:ineq_1}
\bE_G\bqth{\frac{\indc{\sigma\ge \mu}}{\sigma}} \le \frac{c_1\sqrt{t}}{\mu}. 
\end{align}

Next we take $\Delta=1$ in \eqref{eq:lb}. For the denominator, we have
\begin{align}\label{eq:denominator}
\bE_{G}\bpth{\Phi(\frac{1}{\sigma}) - \Phi(0)+\Phi(\frac{1 + \mu}{\sigma}) - \Phi(\frac{\mu}{\sigma})} &\le 2\bE_{G}\bpth{\Phi(\frac{1}{\sigma}) - \Phi(0)} \nonumber\\
&\le 2\pth{\frac{1}{2}G([0,1]) + \Phi'(0)\cdot \bE_G\qth{\frac{\indc{\sigma\ge 1}}{\sigma}}} \nonumber \\
&\le G([0,1]) + \bE_G\bqth{\frac{\indc{\sigma\ge 1}}{\sigma}}. 
\end{align}
For the numerator, we distinguish into two cases. 

\paragraph{Case I: $\mu\ge 1$.} Let
\begin{align*}
\Xi_{\mu,\sigma} := \Phi(\frac{1}{\sigma}) - \Phi(0)  - \Phi(\frac{1 + \mu}{\sigma}) +  \Phi(\frac{\mu}{\sigma}). 
\end{align*}
When $\mu \ge 1$, simple algebra yields that for a universal constant $c_2>0$, 
\begin{align}
\Xi_{\mu,\sigma} \ge c_2\begin{cases}
    1 & \text{if } \sigma < 1, \\
    \frac{1}{\sigma} & \text{if } 1\le \sigma < \mu, \\
    0 & \text{if } \sigma \ge \mu. 
\end{cases}
\end{align}
Therefore,
\begin{align}
\bE_G[\Xi_{\mu,\sigma}] \ge c_2\bpth{G([0,1]) + \bE_G\bqth{\frac{\indc{1\le\sigma\le \mu}}{\sigma}}}, 
\end{align}
and
\begin{align*}
\bpth{G([0,1]) + \bE_G\bqth{\frac{\indc{1\le\sigma\le \mu}}{\sigma}}}^2 &\stepa{\le} c_3t \bpth{G([0,1]) + \bE_G\bqth{\frac{\indc{\sigma\ge 1}}{\sigma}}} \\
&\stepb{\le} c_3t\bpth{G([0,1]) + \bE_G\bqth{\frac{\indc{1\le\sigma\le \mu}}{\sigma}} + \frac{c_1\sqrt{t}}{\mu}}. 
\end{align*}
Here (a) follows from \eqref{eq:lb} and \eqref{eq:denominator}, and (b) uses \eqref{eq:ineq_1}. Solving this quadratic inequality gives
\begin{align*}
G([0,1]) + \bE_G\bqth{\frac{\indc{1\le\sigma\le \mu}}{\sigma}} \le c_4\bpth{t + \frac{t^{3/4}}{\mu^{1/2}}}. 
\end{align*}
As $G([0,1])\ge p \ge Ct$ for $C\ge 2c_4$, this inequality gives that $\mu = O(\frac{t^{3/2}}{p^2})$. 

\paragraph{Case II: $0\le \mu<1$.} In this case, one can verify the following lower bounds for $\Xi_{\mu,\sigma}$: 
\begin{align}
\Xi_{\mu,\sigma} \ge c_2\begin{cases}
    1 & \text{if } \sigma < \mu, \\
    \frac{\mu}{\sigma} & \text{if } \mu\le \sigma < 1, \\
    0 & \text{if } \sigma \ge 1. 
\end{cases}
\end{align}
Therefore, the numerator is lower bounded as
\begin{align}
\bE_G[\Xi_{\mu,\sigma}] \ge c_2\bpth{ G([0,\mu]) +  \bE_G\bqth{\frac{\mu\indc{\mu\le\sigma\le 1}}{\sigma}} } \ge c_2\mu \cdot G([0,1]). 
\end{align}
Consequently, by \eqref{eq:lb}, \eqref{eq:ineq_1}, and \eqref{eq:denominator}, we get
\begin{align*}
\mu^2 G([0,1])^2 \le c_3t\bpth{G([0,1]) + \bE_G\bqth{\frac{\indc{\sigma\ge 1}}{\sigma}}} \le c_3t\bpth{G([0,1]) + \frac{c_1\sqrt{t}}{\mu}}, 
\end{align*}
and further solve this quadratic inequality to find
\begin{align*}
G([0,1]) \le c_4\bpth{\frac{t}{\mu^2} + \frac{t^{3/4}}{\mu^{3/2}}}. 
\end{align*}
Finally, as $G([0,1])\ge p$, we obtain $\mu = O(\frac{t^{1/2}}{p^{1/2}} + \frac{t^{1/2}}{p^{2/3}}) = O(\frac{t^{1/2}}{p^{2/3}})$ as $p\le 1$. 

% \paragraph{Summary.} 
To reach the desired conclusion, note that when $p^{4/3}<t\le \frac{p}{C}$, we have either $\mu = O(\frac{t^{3/2}}{p^2})$ (Case I), or $\mu\le 1$ (assumption of Case II). Since the upper bound in Case I is larger, we conclude that $\mu = O(\frac{t^{3/2}}{p^2})$. Similarly, when $t\le p^{4/3}$, we have either $\mu = O(\frac{t^{3/2}}{p^2})$ (Case I), or $\mu=O(\frac{t^{1/2}}{p^{2/3}})$ (Case II). Here the upper bound in Case II dominates, so that $\mu=O(\frac{t^{1/2}}{p^{2/3}})$. 

\subsection{Modulus of Continuity for Selected Priors: Proof of \Cref{cor:examples}}
\label{sec:mod_cont}

To prove \Cref{cor:examples}, we invoke \Cref{thm:general} and establish upper bounds of the Hellinger modulus of continuity for various choices of $G_n$. 

\paragraph{Equal variance.} When $\sigma_i \equiv 1$, we have $G_n = \delta_{1}$. In this case, 
\begin{align*}
H^2(f_{\mu_1, G_n}, f_{\mu_2, G_n}) = 2-2\exp\pth{-\frac{(\mu_1-\mu_2)^2}{8}}, 
\end{align*}
so that $\omega_{H^2,G_n}(t) = O(\sqrt{t})$. Now plugging in $t = O(\frac{L}{n})$ gives the result. 

\paragraph{Quadratic variance.} Consider any $\mu > 0$ such that \eqref{eq:lb} holds for all $\Delta\ge 0$, with $G = G_n$. Next we choose $\Delta = \mu$ in \eqref{eq:lb} and proceed in the same way as \eqref{eq:denominator} to obtain
\begin{align*}
G_n([0,\mu])^2 \le c_1t\pth{ G_n([0,\mu]) + \mu \bE_{G_n}\bqth{\frac{\indc{\sigma\ge \mu}}{\sigma}} }. 
\end{align*}
For $t\le c_2$ with a small enough constant $c_2>0$, the above inequality shows that $\mu > n$ is impossible. For $1\le \mu\le n$, since $\bE_{G_n}[\frac{\indc{\sigma\ge \mu}}{\sigma}] \le \bE_{G_n}[\frac{1}{\sigma}] = \frac{1}{n}\sum_{i=1}^n \frac{1}{i} = O(\frac{\log n}{n})$ and $G_n([0,\mu])\asymp \frac{\mu}{n}$, we can solve that $\mu = O(nt\log n)$. Combining with the remaining case $\mu\le 1$ leads to the final upper bound
\begin{align*}
\omega_{H^2,G_n}(t) = O\pth{ nt\log n \vee 1 }, \quad \text{if } t \le c_2. 
\end{align*}
Plugging in $t = O(\frac{L}{n})$ proves the target result. 

\paragraph{Two variances.} Again, consider some $\mu > 0$ such that \eqref{eq:lb} holds for every $\Delta\ge 0$, with $G = G_n = \frac{m}{n}\delta_1 + (1-\frac{m}{n})\delta_{\sigma}$. We consider two choices of $\Delta$. First, choosing $\Delta = \sigma$ and upper bounding the denominator by $2$ in \eqref{eq:lb}, we get
\begin{align*}
   \frac{n-m}{n}\pth{\Phi(1) - \Phi(0) - \Phi(1+\frac{\mu}{\sigma}) + \Phi(\frac{\mu}{\sigma})} \le c_1\sqrt{t}. 
\end{align*}
Clearly, for $t\le c_2$ with a small enough constant $c_2>0$, this inequality implies that $\mu \le \sigma$. In this case, the left-hand side further scales as $\Omega(\frac{\mu}{\sigma})$, so that we get
\begin{align*}
\omega_{H^2,G_n}(t) = O(\sigma \sqrt{t}), \quad \text{if } t \le c_2. 
\end{align*}
For $t=O(\frac{L}{n})$, this gives the general $O(\sigma\sqrt{\frac{L}{n}})$ upper bound.

The second choice of $\Delta$ is $\Delta = 1$. If we upper bound the denominator of \eqref{eq:lb} by $2$, then
\begin{align*}
   \frac{m}{n}\pth{\Phi(1) - \Phi(0) - \Phi(1+\mu) + \Phi(\mu)} \le c_1\sqrt{t}. 
\end{align*}
Therefore, if $\frac{n\sqrt{t}}{m}\le c_2$ for a small enough constant $c_2>0$, we solve that $\mu = O(\frac{n\sqrt{t}}{m})$. For $t=O(\frac{L}{n})$, this proves the $O(\frac{\sqrt{nL}}{m})$ upper bound for $m\ge \sqrt{nL}$. Instead, if we upper bound the denominator of \eqref{eq:lb} via \eqref{eq:denominator}, we obtain
\begin{align*}
   \frac{m}{n}\pth{\Phi(1) - \Phi(0) - \Phi(1+\mu) + \Phi(\mu)} \le c_1\sqrt{t\pth{\frac{m}{n} + \frac{1}{\sigma}}}. 
\end{align*}
For $t=O(\frac{L}{n})$ and $m\ge C(\sqrt{\frac{nL}{\sigma}} + L)$, it holds that
\begin{align*}
\frac{m}{n}\ge 10c_1 \sqrt{t\pth{\frac{m}{n} + \frac{1}{\sigma}}}. 
\end{align*}
Solving the inequality gives the target upper bound $\mu = O(1)$. 

\paragraph{$\alpha$-mixture distributions.} We use the upper bound in the two variances example. For $\alpha\in (0,1)$, we use the upper bound $O(\sigma\sqrt{\frac{L}{n}}) = \widetilde{O}(n^{\alpha-1/2}\sqrt{L})$. For $\alpha\ge 1$, we note that for $m=cL$ with a large enough constant $c>0$, it holds that
\begin{align*}
C\bpth{\sqrt{\frac{nL}{n^{\alpha}}} + L} \le m < \sqrt{nL}. 
\end{align*}
Therefore, we can apply the $O(1)$ upper bound in this case.

\section{Additional Details of \Cref{sec:numerics}}

\subsection{Proof of \Cref{lemma: structural}}

Since $X_i\neq 0$ for all $i\in [n]$, the function $G\mapsto f_{0,G}(X_i)$ has a finite upper bound depending only on $X_i$, so the log-likelihood $\sum_{i=1}^n \log f_{0,G}(X_i)$ cannot reach $+\infty$. Then standard compactness argument shows the existence of the NPMLE $\widehat{G}$, and by the strict concavity of $x\mapsto \log x$, the vector of densities $(f_{0,\widehat{G}}(X_1),\dots,f_{0,\widehat{G}}(X_n))$ is unique. Denote it by $(f_1,\dots,f_n)$. 
    
By the KKT condition \eqref{eq:KKT}, $\supp(\widehat{G})$ is a subset of the set of global maximizers of $\sigma\mapsto D_{\widehat{G}}(\sigma)$. By differentiation, the set of critical points of $\sigma\mapsto D_{\widehat{G}}(\sigma)$ is
\begin{align}\label{eq: support}
    C := \bsth{ \sigma > 0 : \frac{1}{n}\sum_{i=1}^n \frac{f_i}{\sqrt{2\pi}}e^{-X_i^2/2\sigma^2}\bpth{\frac{X_i^2}{\sigma^2} - 1} = 0}.
\end{align}
Clearly, $C\subseteq [\min_i|X_i|, \max_i |X_i|]$, so the same holds for $\supp(\widehat{G})$. By combining repeated appearances of $|X_i|$, WLOG we assume that $|X_1|,\dots,|X_n|$ are distinct. Next we show that $C$ is a finite set, with $|C|\le 2n-1$. Therefore, for some distinct constants $c_1,c_2,\dots,c_n$ and a non-zero vector $(d_1,\dots,d_{2n})$, 
\begin{align*}
\bsth{\frac{1}{\sigma^2}: \sigma\in C} \subseteq \bsth{x>0: \sum_{i=1}^n( d_i + d_{n+i}x)e^{c_ix}  = 0 }.
\end{align*}
To proceed, we recall the following result taken from \citep[Page 48]{polya1925aufgaben}. 

\begin{lem}\label{lemma:chebyshev}
Let $c_1,\dots,c_n\in \bR$ be distinct, and $P_i(x)$ be a polynomial in $x$ with degree $m_i-1$. Then the equation $\sum_{i=1}^n P_i(x)e^{c_ix}=0$ has at most $(\sum_{i=1}^n m_i)-1$ real solutions. 
\end{lem}

By \Cref{lemma:chebyshev}, the above equation has at most $2n-1$ solutions. This shows that $|C|\le 2n-1$. Finally, since the continuous map $\sigma\mapsto D_{\widehat{G}}(\sigma)$ must have at least one critical point between two global maximizers, we conclude from the upper bound on $|C|$ that $|\supp(\widehat{G})| \le n$. 

Finally we show the uniqueness of $\widehat{G}$. Since the map $\sigma\mapsto D_{\widehat{G}}(\sigma)$ only depends on $\widehat{G}$ through the likelihood vector $(f_1,\dots,f_n)$, the set $D$ of its global maximizers is a fixed set for all versions of $\widehat{G}$. In addition, $|D|\le n$, so we can write $D=\{\sigma_1,\dots,\sigma_m\}$ with $m\le n$, and any NPMLE $\widehat{G}$ takes the form $\sum_{j=1}^m w_j\delta_{\sigma_j}$. It remains to show the uniqueness of $(w_1,\dots,w_m)$. Now by the uniqueness of $(f_1,\dots,f_n)$, $(w_1,\dots,w_m)$ is a solution to the linear system
\begin{align*}
\sum_{j=1}^m w_j \frac{1}{\sigma_j}\varphi(\frac{X_i}{\sigma_j}) = f_{0,\widehat{G}}(X_i) = f_i, \quad \forall i\in [n].
\end{align*}
To prove uniqueness, it suffices to show that the matrix $A=(A_{ij})_{i\in [n], j\in [m]}$ with $A_{ij} = \frac{1}{\sigma_j}\varphi(\frac{X_i}{\sigma_j})$ has full column rank. By scaling the columns and taking $m\times m$ submatrices, it further suffices to prove that $B=(B_{ij})\in \bR^{m\times m}$ with $B_{ij} = \varphi(\frac{X_i}{\sigma_j})$ has full rank. Assuming the contrary, then $Bz=0$ would have a non-zero solution $z\in \bR^m$, and the map
\begin{align*}
t \mapsto \sum_{j=1}^m \exp(-\frac{t}{\sigma_j^2})z_j 
\end{align*}
would have at least $m$ distinct zeros. Since $\sigma_1,\dots,\sigma_m > 0$ are distinct, this is a contradiction to the statement of \Cref{lemma:chebyshev}. This concludes the uniqueness of $\widehat{G}$.

\subsection{Additional Experimental Details in \Cref{sec:numerics}}\label{appendix: extra simul}
We provide additional experimental details related to our implementation, as well as figures for the equal and quadratic variance cases of heteroskedastic mean estimation. 
For all experiments, we set $\sigmin = 0$ and $\sigmax = \infty$, but in general one can initialize with prespecified values for $(\sigmin, \sigmax)$. 
For every instance of the Frank--Wolfe algorithm (i.e., estimating $\widehat{G}$ with a fixed $\mu$), we use \Cref{lemma: structural} to set a data-driven support $I = [\min_i |X_i-\mu|, \max_i |X_i-\mu|]$ for the new atoms $\sigma_t$, and apply a grid search over $N=5,000$ points in $I$ to find the location of the new atom $\sigma_t$. In addition, we initialize the atoms of $\widehat{G}_0$ to be $5$ evenly distributed points in $I$. When finding the MLE estimator $\widehat{\mu}$ given the current $\widehat{G}$, we also apply a grid search over $N=5,000$ points in the interval $[\min_{i\in[n]} X_i, \max_{i\in [n]} X_{i}]$. 
%For every instance of the successive maximization algorithm, we ensure that any new atom added lies in the prespecified interval $[\sigma_{\min}, \sigma_{\max}]$ %\YH{this sentence suggests that we're using different $\sigmin, \sigmax$ at every step? why not combine with the next sentence and provide a detailed specification.}% 
%and truncate them if they lie outside the data-dependent values $\min_{i\in [n]}|X_{i}|$ and $\max_{i\in[n]}|X_{i}|$ (see \Cref{lemma: structural}). 
% At initialization,  For both finding the MLE estimator $\widehat{\mu}_t$ given the current iterate $\widehat{G}_t$ and updating the new atom $\sigma_t \in \argmax_{\sigma > 0}D_{\widehat{G}_{t-1}}(\sigma; X^{n}-\mu)$, we rely on a fine grid search over $N=5000$ points in the interval $[\min_{i\in[m]} X_i, \max_{i\in [m]} X_{i}]$. %While computationally expensive, we have found that our empirically works well in practice. 
%We also set the initial grid for the prior in the alternating Frank--Wolfe algorithm to consist of $n=5$ evenly distributed points in $[\sigma_{\min}, \sigma_{\max}]$, with uniform initial weights. 
Therefore, aside from typical tolerance thresholds, our algorithm is completely tuning-parameter free. Code for our implementation and simulation results can be found in the repository \url{https://github.com/AshettyV/NPMLE}, and we also include the experimental results for the cases of equal and quadratic variance in \Cref{cor:examples}.

\begin{figure}[ht]
    \centering
    \subfigure[$P_1 = \dots = P_n = \mathcal{N}(0,5)$]{
    \includegraphics[width=0.50\linewidth]{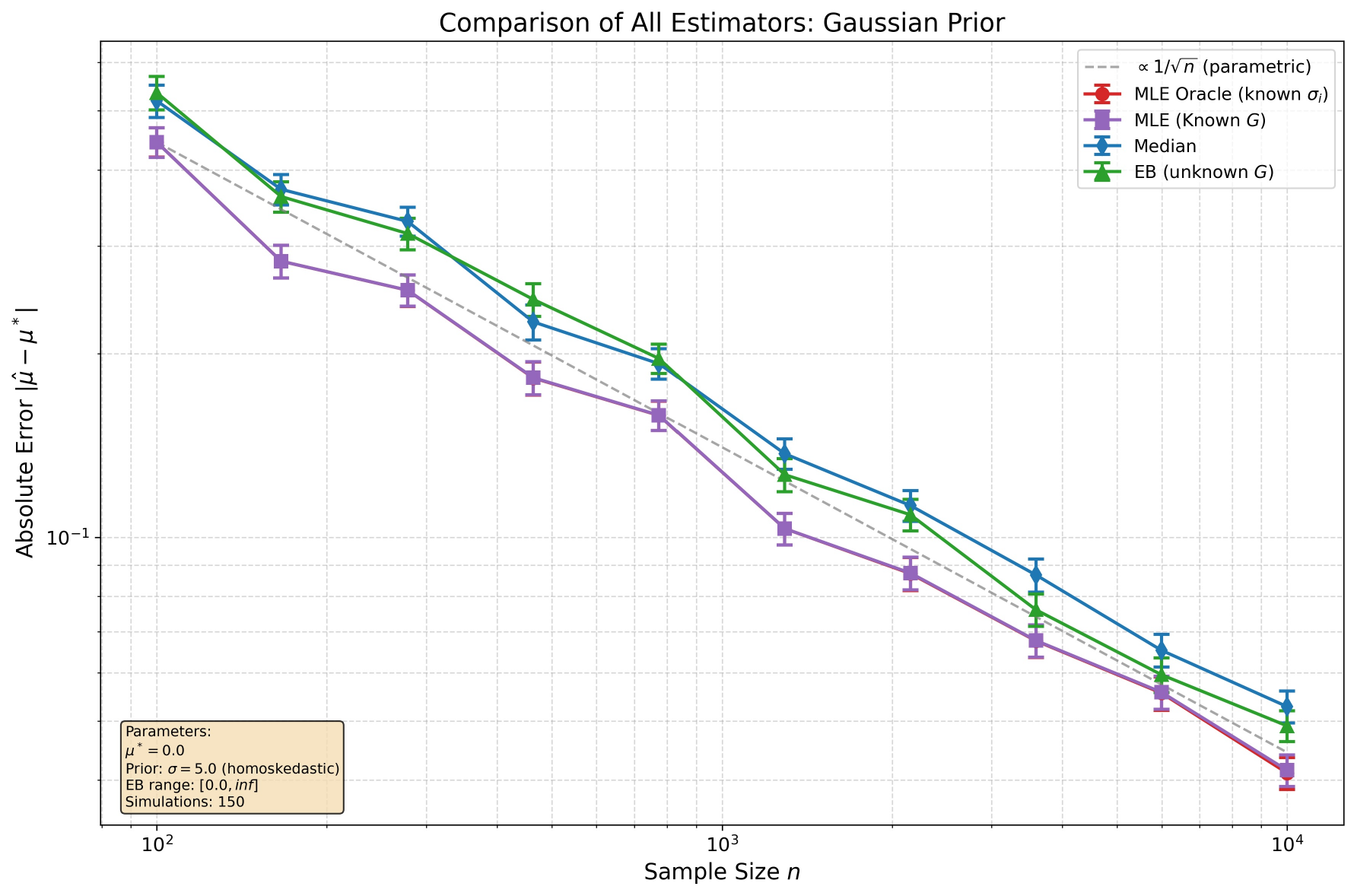}}
    % \hfill
    \subfigure[$P_i = \mathcal{N}(0, \frac{i^2}{10}), \quad 1\leq i \leq n$]{\includegraphics[width=0.48\linewidth]{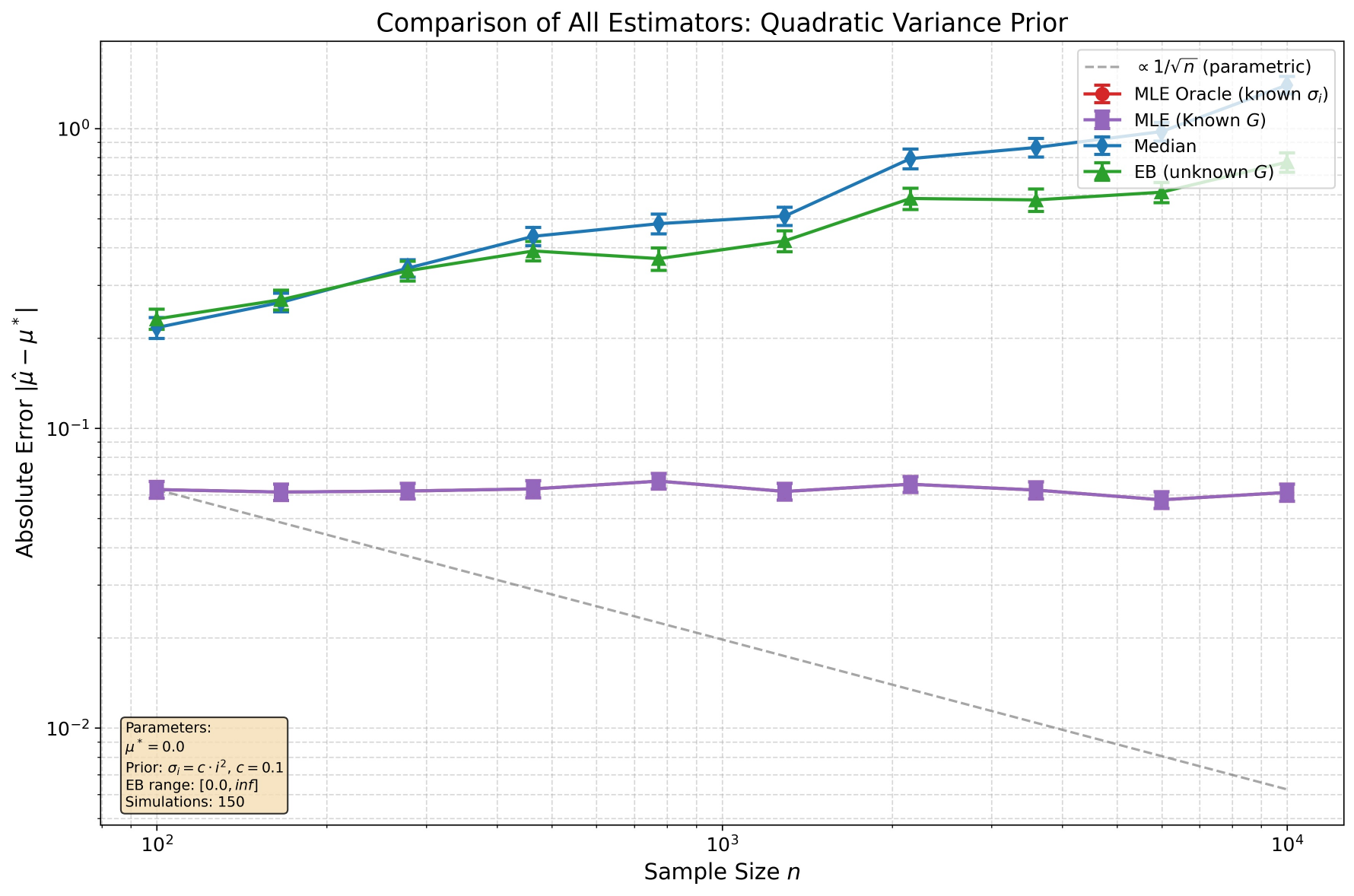}}
    \caption{Average absolute estimation error for $\widehat{\mu}^{\text{EB}}$, sample median, and two oracle MLEs over $N=150$ simulations, under the equal variance model where $P_1 = \dots = P_n = \mathcal{N}(0,5)$ and the quadratic variance model $P_i = \mathcal{N}(0, \frac{i^2}{10})$ for $i\in [n]$.}
\end{figure}

% \section{Chebyshev Approximation}
% \input{app_chebyshev.tex}

% \section{Additional Numerical Experiments}

% \input{app_numerical.tex}

\end{document}